\documentclass[a4paper, 15pt]{article}
\usepackage[utf8]{inputenc}
\usepackage[T1]{fontenc}
\usepackage[french,english]{babel}    
\usepackage{tikz}
\usetikzlibrary{positioning}

\usepackage[left=2cm,right=2cm,top=2cm,bottom=2cm]{geometry}         		
\usepackage{graphicx}				
\usepackage{amssymb}
\usepackage[backend=bibtex,style=numeric,
          maxnames=1000]{biblatex}
\usepackage{lineno}
\usepackage{amsthm}
\usepackage{amsmath}
\usepackage{amssymb}
\usepackage{amsfonts}
\usepackage{mathtools}
\usepackage{amssymb}
\usepackage{latexsym}
\usepackage{hyperref}
\usepackage{mathrsfs}
\usepackage[utf8]{inputenc}
\usepackage[parfill]{parskip}
\usepackage{xcolor}
\usepackage[affil-it]{authblk}
\usepackage{enumitem}

\DeclarePairedDelimiter\floor{\lfloor}{\rfloor}
\usepackage[left=2cm,right=2cm,top=2cm,bottom=2cm]{geometry}

\numberwithin{equation}{section}

\newtheorem{defi}{Definition}[section]
\newtheorem{unTheorem}{Theorem}[section]

\newtheorem{propal}[defi]{Proposition}
\newtheorem{rem}[defi]{Remark}
\newtheorem{cor}[defi]{Corollary}

\newtheorem{nota}[defi]{Notation}

\title{Semi-classical limit of the Klein-Gordon equation to relativistic Euler equations via an adapted modulated energy method}
\author[]{Tony Salvi}
\date{}
\begin{document}

\maketitle

\begin{center}
    \textbf{Abstract}
    \hfill\begin{minipage}{\dimexpr\textwidth-1cm}
    
	We show the convergence of the solutions to the massive nonlinear Klein-Gordon equation toward solutions to a relativistic Euler with potential-type system in the semi-classical limit. In particular, the momentum and the density of Klein-Gordon converge to the momentum and the density of the relativistic Euler system in Lebesgue norms. The relativistic Euler with potential system is equivalent to the standard relativistic barotropic Euler system, differing only by a change in the unknowns. The proof relies on the modulated energy method adapted to the wave equation and the relativistic setting: a modulated stress-energy method. 
\end{minipage}
\end{center}
\tableofcontents
\section{Introduction} 
\label{section:Intro}
In this paper, we are interested in the semi-classical version of the massive nonlinear Klein-Gordon (KG) equation with a potential
\begin{equation}
\label{eq:KGPintro}	\varepsilon^2\Box\Phi^\varepsilon=\Phi^\varepsilon+2V'(|\Phi^\varepsilon|^2)\Phi^\varepsilon.
\end{equation}
We set ourselves in the Minkowski spacetime, in $\mathbb{R}^{1+3}$ endowed with the Minkowski metric $g=Diag(-1,1,1,1)$. The complex function $\Phi^\varepsilon$ is the wave function, $V'$ is a potential that we restrict to a defocusing power-law nonlinearity $V'(x)=x^{\gamma-1}$ for $\gamma\geq2$, and $\varepsilon$ is a small parameter. In particular, we look at the behaviour of the solutions when $\varepsilon$ goes to 0, in the semi-classical limit. We show that in this limit the dynamics of solutions to \eqref{eq:KGPintro} is given by a relativistic Euler with potential-type system (REP)
\begin{equation}
\label{eq:EulerREPintro}
\begin{cases}
 \textbf{U}^\alpha\textbf{D}_\alpha \textbf{U}_\beta+\textbf{D}_\beta V'(\rho)=0,\\
\textbf{U}^\alpha\textbf{D}_\alpha \rho+\textbf{D}_\alpha \textbf{U}^\alpha\rho=0,\\
\textbf{U}^\alpha \textbf{U}_\alpha+2V'(\rho)=-1,
\end{cases}
\end{equation}
where $\textbf{U}$ is a four-velocity vector field, $\rho$ a density, and $V'$ the same as in \eqref{eq:KGPintro}. In particular, the momentum $\textbf{J}^\varepsilon=-\Im(\Phi^\varepsilon\overline{\varepsilon \textbf{D}\Phi^\varepsilon})$ and the density $\rho^\varepsilon=|\Phi^\varepsilon|^2$ of \eqref{eq:KGPintro} converge to the momentum $\textbf{J}=\textbf{U}\rho$ and the density $\rho$ associated with \eqref{eq:EulerREPintro} as $\varepsilon\to0$. The REP system \eqref{eq:EulerREPintro} is an unusual model for relativistic fluids. We derive it heuristically from the WKB method applied to \eqref{eq:KGPintro} in section \ref{subsection:WKB} to make its emergence clearer. In fact, the REP system \eqref{eq:EulerREPintro} is of particular interest because it is equivalent to the more usual relativistic barotropic Euler 
\begin{equation}
    \begin{cases}
    \label{eq:relatEulintro}
        \textbf{u}^\alpha\textbf{D}_\alpha \textbf{u}^\beta+(g^{\alpha\beta}+\textbf{u}^\alpha \textbf{u}^\beta)\frac{\textbf{D}_\alpha p}{\mu+p}=0,\\
        \textbf{u}^\alpha\textbf{D}_\alpha\mu+\textbf{D}_\alpha \textbf{u}^\alpha(\mu+p)=0,\\
        p=p(\mu),\\
        \textbf{u}^\alpha \textbf{u}_\alpha=-1,
    \end{cases}
\end{equation}
where $\textbf{u}$ is also a four-velocity vector field, $\mu$ a density, and $p$ corresponds to a pressure. The equivalence is understood
in the sense that, with the right change of unknown, $(\textbf{u},\mu,p)=(\textbf{u}(\textbf{U},\rho,V'(\rho)),\mu(\textbf{U},\rho,V'(\rho)),p(\textbf{U},\rho,V'(\rho)))$ is a solution to \eqref{eq:relatEulintro}
if and only if $(\textbf{U},\rho)$  is a solution to \eqref{eq:EulerREPintro}. This implies that the semi-classic limit of KG is the usual relativistic barotropic Euler. We give a discussion on the relation between \eqref{eq:EulerREPintro} and \eqref{eq:relatEulintro} and their non-relativistic limits in sections \ref{section:relatEul}.\\\\
The first inspirations of this work are the book \cite{zbMATH05243173} (and notably the chapter on the modulated energy method that presents arguments from \cite{zbMATH05503664}, inspired by \cite{zbMATH05000052}) on the semi-classical limit of Schrödinger and the sequence of papers \cite{zbMATH05782734}, \cite{zbMATH06101438} and \cite{zbMATH06039446} on the hydrodynamic limits of Klein-Gordon. Indeed, these works are linked because the Klein-Gordon equation is a relativistic "equivalent" of the Schrödinger equation, see \cite{zbMATH01837395,zbMATH01747314,zbMATH03871951,NAJMAN1990217,zbMATH06303453,zbMATH01650142,pasquali2018dynamicsnonlinearkleingordonequation,lei2023nonrelativistic,bambusi2025nonrelativisticlimitnonlinear}
for rigorous proofs of the non-relativistic limit of Klein-Gordon to Schrödinger. Nonetheless, mathematically speaking, these two equations are quite different and the methods of \cite{zbMATH05243173} must be adapted to the wave equation and the relativistic setting. The results of \cite{zbMATH06101438} find themselves in the middle of the relativistic and the non-relativistic setting and do not exhibit the structure we show in this paper. Moreover, the proof of the semi-classical limit of Klein-Gordon in \cite{zbMATH06101438} drastically relies on a rescaling in time and strong uniform assumptions on the energy, this masks the true relativistic nature of the hydrodynamic limit system. More details on the consistency of our results and theirs are given in section \ref{subsection:comments}. 
\\\\
As done in the previously cited papers, we rely on the  modulated energy method. This method goes back to \cite{zbMATH01442880}, on the quasineutral limit for Vlasov-Poisson, and was quickly adapted to the semi-classical limit of the Schrödinger equation, for example in \cite{zbMATH01877186}, \cite{zbMATH01970480}, \cite{zbMATH01987564}, \cite{zbMATH05000052}, then \cite{zbMATH05243173} and more recently \cite{zbMATH07570763}. See \cite{zbMATH06101438} for an overview of the history of this method on Schrödinger-affiliated systems. The main idea of the modulated energy method is to construct a functional $H^\varepsilon$ (the modulated energy) that depends on the time $t$ and that controls the difference between quantities of interest (here the momentum and the density) of a starting equation (here KG \eqref{eq:KGPintro}) and the same quantities for an arrival system (here the REP system \eqref{eq:EulerREPintro}) in  $L^p$ norms. Moreover, under suitable conditions on the initial data (for \eqref{eq:KGPintro} and \eqref{eq:EulerREPintro}, given in definition \ref{defi:wellprepREP} and in the assumptions of Theorem \ref{unTheorem:TH1mainth}), the modulated energy must be such that if $H^\varepsilon(0)=O(\varepsilon^2)$ then $H^\varepsilon(t)=O(\varepsilon^2)$ for $t$ in some interval $[0,T]$. These properties imply the convergence of the quantities of interest at the limit $\varepsilon\to0$.\\ Considering the standard modulation of the energy for the nonlinear Schrödinger equation of  \cite{zbMATH05243173} (inspired by \cite{zbMATH05000052}), the expected modulated energy for the nonlinear Klein-Gordon equation must be 
\begin{equation}
\label{eq:moduH0}
     H^\varepsilon_0=\int{\frac{|(\varepsilon\textbf{D}-i\textbf{U})\Phi^\varepsilon|^2}{2} +\left(V(|\Phi^\varepsilon|^2)-V(\rho)-V^{\prime}(\rho)(|\Phi^\varepsilon|^2-\rho)\right)dx}.
\end{equation}
The actual modulated energy we build satisfies 
\begin{align}
    H^\varepsilon\sim H^\varepsilon_0
\end{align}
where the equivalence is understood in the sense of norms. 
The use of $H^\varepsilon$, instead of $H^\varepsilon_0$, is needed in our proof and the construction of the latter requires the modulation of the full stress-energy-tensor of KG
\begin{align}
\label{eq:stressenergintro}
    T^{KG}[\Phi^\varepsilon]_{\alpha\beta}=\frac{\varepsilon^2}{2}(\textbf{D}_\alpha\Phi^\varepsilon\textbf{D}_\beta\overline{\Phi^\varepsilon}+\textbf{D}_\alpha\overline{\Phi^\varepsilon}\textbf{D}_\beta\Phi^\varepsilon)-\frac{1}{2}g_{\alpha\beta}(\varepsilon^2\textbf{D}_\mu\overline{\Phi^\varepsilon}\textbf{D}^\mu\Phi^\varepsilon+|\Phi^\varepsilon|^2+2V(|\Phi^\varepsilon|^2)),
\end{align}
it is properly defined in section \ref{section:defmod}.
In particular, the modulated "stress-energy" $H^\varepsilon$ takes into consideration the full relativistic aspect of the problem and obeys better evolution equations than $H^\varepsilon_0$. We note that we adapt the modulated energy of \cite{zbMATH05243173} directly for wave equations and the relativistic setting, while in \cite{zbMATH06101438} the authors use a modulated energy on the "modulated" Klein-Gordon equation, i.e., a Schrödinger equation with a source term whose unknown is, up to a rescaling, the wave function solution to \eqref{eq:KGPintro}. As far as we know, our work is the first to derive relativistic fluid equations in a semi-classical limit. However, in the semi-relativistic setting, the semi-classical limit of Pauli-Poisswell is treated in \cite{yang2024semiclassicallimitpaulipoisswelleulerpoisswell} to obtain the Euler-Poisswell system using WKB techniques.  \\
The arguments given in our proof rely on the existence of sufficiently regular and global solutions to the massive nonlinear Klein-Gordon equation. Such solutions exist in the subcritical and critical cases. This is a standard result, see \cite{Brenner1979,Wahl1981,zbMATH03875887,zbMATH04202876,zbMATH00011219,zbMATH00168350,zbMATH00703984,zbMATH00817689,zbMATH05035890,luk2014introduction} for example. In dimension 3, the critical exponent corresponds to $\gamma=3$.
\subsection{Context on the Klein Gordon equation}
\label{subsection:contextKG}
We introduce the massive nonlinear Klein-Gordon equation in its physical scaling 
\begin{equation}
\label{eq:KGPphysicalcontext}	\hbar^2\Box\Phi=c^2m^2\Phi+2V'(|\Phi|^2)m\Phi,
\end{equation}
with $c$ the speed of light, $\hbar$ the Planck constant divided by $2\pi$, $m>0$ the mass (set to 1 in what follows) and $\Box=-\frac{1}{c^2}\partial_{tt}^2+\Delta$ the d'Alembertian. In the absence of a potential, the Klein-Gordon equation is directly derived via the relativistic dispersion relation (the energy-momentum relation $\textbf{p}^\alpha\textbf{p}_\alpha=-m^2c^4$) as the Schrödinger equation is derived via the non-relativistic dispersion relation for the mass, the energy and the momentum  ($\frac{p^2}{2m}=E$). The wave function solution to \eqref{eq:KGPphysicalcontext} with $V'=0$ describes the evolution of a relativistic massive spinless free particle. We add the potential to model more complex phenomena.
By setting $c=1$ and $\varepsilon=\hbar$ we recover \eqref{eq:KGPintro}. Indeed, in the relativistic context $c$ is normalized to 1 and in semi-classical physics $\hbar$ is considered a very small constant.\\
One can also recover \eqref{eq:KGPintro} using the natural units $c=\hbar=1$ and rescaling the solution $\Phi$ to \eqref{eq:KGPphysicalcontext} as $\Phi^\varepsilon(t,x)=\Phi(\frac{t}{\varepsilon},\frac{x}{\varepsilon})$. Thus, the solutions $\Phi^\varepsilon$  to \eqref{eq:KGPintro} are intrinsically high frequency, see section \ref{subsection:WKB} for more details.\\
With the previous remark, we see that the solution $\Phi$ is rescaled so that the time interval on which the $\Phi^\varepsilon$ are defined shrinks as $\varepsilon$ approaches 0. This is not a problem for global solutions, if $\Phi$ is defined on $\mathbb{R}_+\times\mathbb{R}^3$ then so is $\Phi^\varepsilon$ for any $\varepsilon>0$. It is important because we want the convergence of  $\textbf{J}^\varepsilon$ (resp. $\rho^\varepsilon$) to $\textbf{J}$ (resp. $\rho$) to make sense on a fixed time interval (uniform in $\varepsilon$). 
In other words, we want the family $(\Phi^\varepsilon)_{0<\varepsilon<1}$ of solutions to \eqref{eq:KGPintro} and the solution to the REP system to exist on a common time interval with enough regularity to manage the proof of convergence. \\
From a mathematical point of view, the Klein-Gordon equation is a nonlinear wave equation. Thus, a rich literature can be used to study it. See \cite{zbMATH05035890}, \cite{wang2015lectures} or \cite{luk2014introduction} for example.\\
\subsection{WKB approximation}
\label{subsection:WKB}
The solutions to \eqref{eq:KGPintro} are truly high frequency because of the scaling of the equation. It is natural to use the WKB expansion to get approximate solutions and to understand their behaviours. We refer to \cite{METIVIER2009169}, \cite{zbMATH06039454}, and the lecture notes \cite{Rauch1999LecturesOG} (with applications on Maxwell's equations) for a general presentation of the WKB method and to \cite{zbMATH01132286} for an application on the semi-classical Schrödinger equation. The WKB method does not close in our case\footnote{In particular, the techniques of \cite{zbMATH01132286}, which are build on specific structures of the nonlinear Schrödinger equation, do not apply to the nonlinear Klein-Gordon equation.}, and we use the modulated energy method to prove our result presented in the next section \ref{subsection:mainresults}. Nonetheless, the WKB method is easier to read, so we use it heuristically to motivate our result. We look for $(\Phi^\varepsilon_1)_{0<\varepsilon<1}$, a family of approximate solutions, under the form of a phase-amplitude ansatz
\begin{equation}
    \Phi_1^\varepsilon=e^{i\frac{\omega}{\varepsilon}}A
\end{equation}
with $\omega$ a smooth real phase and $A$ a smooth complex amplitude. In particular, the high-frequency character of $\Phi^\varepsilon_1$ is explicit and concentrated in the phase.  \\
\begin{propal}
\label{propal:approxWKB}
    Let $\omega$ and $A$ satisfy
    \begin{equation}
    \label{eq:approxsystWKB}
\begin{cases}
    \textbf{D}^\alpha \omega\textbf{D}_\alpha \omega=-1-2V'(|A|^2),\\
    2\textbf{D}^\alpha \omega\textbf{D}_\alpha A+A\Box \omega=0.
    \end{cases}
\end{equation}
Then, $\Phi_1^\varepsilon$ is an approximate solution of order $2$ to \eqref{eq:KGPintro}, that is
\begin{equation}   
\label{eq:approxsolutWKB}
\varepsilon^2\Box\Phi_1^\varepsilon-\Phi_1^\varepsilon-2V'(|\Phi_1^\varepsilon|^2)\Phi_1^\varepsilon=O(\varepsilon^2).
\end{equation}
\end{propal}
\begin{proof}
    By direct calculations we get 
    \begin{equation}   
\varepsilon^2\Box\Phi_1^\varepsilon-\Phi_1^\varepsilon-2V'(|\Phi_1^\varepsilon|^2)\Phi_1^\varepsilon=e^{i\frac{\omega}{\varepsilon}}(-\textbf{D}^\alpha \omega\textbf{D}_\alpha \omega-1-2V'(|A|^2))+\varepsilon ie^{i\frac{\omega}{\varepsilon}}(2\textbf{D}^\alpha \omega\textbf{D}_\alpha A+A\Box \omega)+\varepsilon^2e^{i\frac{\omega}{\varepsilon}}(\Box A)=O(\varepsilon^2).
\end{equation}
\end{proof}

\begin{rem}
\label{rem:eikonalWKB}
The equation 
\begin{equation}
    \label{eq:eikonalWKB}
    \textbf{D}^\alpha \omega\textbf{D}_\alpha \omega=-1-2V'(|A|^2)\\
\end{equation}
is called the eikonal equation. In the case of plane waves, $\omega=x^\alpha \textbf{k}_\alpha$ with $\textbf{k}$ and $A$ constant, we recover the dispersion relation 
\begin{equation}
    \label{eq:dispersionWKB}
    \textbf{k}^\alpha \textbf{k}_\alpha=-1-2V'(|A|^2).\\
\end{equation}
\end{rem}

\begin{propal}
    Let $(\omega,A)$ be solution to \eqref{eq:approxsystWKB}, then $(\textbf{U}_\alpha,\rho)=(\textbf{D}_\alpha \omega,|A|^2)$ is solution to 
     \begin{equation}
    \label{eq:approxsysteulerWKB}
\begin{cases}
    \textbf{U}^\alpha\textbf{D}_\alpha \textbf{U}_\beta+\textbf{D}_\beta V'(\rho)=0,\\
    \textbf{U}^\alpha\textbf{D}_\alpha \rho+\rho\textbf{D}_\alpha \textbf{U}^\alpha=0,
    \end{cases}
\end{equation}
and $\textbf{U}$ is a timelike vector field with the normalization 
\begin{equation}
   \label{eq:eikonalnormaliz}
    \textbf{U}^\alpha \textbf{U}_\alpha=-1-2V'(\rho).
\end{equation}
\end{propal}
\begin{proof}
    By direct calculations, we use \eqref{eq:eikonalWKB} to get 
    \begin{align}
          \textbf{U}^\alpha\textbf{D}_\alpha \textbf{U}_\beta+\textbf{D}_\beta V'(\rho)=\textbf{D}^\alpha \omega\textbf{D}_\alpha\textbf{D}_\beta \omega+\textbf{D}_\beta V'(|A|^2)=\frac{1}{2}\textbf{D}_\beta(\textbf{D}^\alpha \omega\textbf{D}_\alpha \omega+1+2V'(|A|^2))=0,
    \end{align}
    and, with the second equation of \eqref{eq:approxsystWKB}, 
     \begin{align}
          \textbf{U}^\alpha\textbf{D}_\alpha \rho+\rho\textbf{D}_\alpha \textbf{U}^\alpha= \textbf{D}^\alpha \omega\textbf{D}_\alpha |A|^2+|A|^2\Box \omega=A(\textbf{D}^\alpha \omega\textbf{D}_\alpha \overline{A}+\frac{1}{2}\overline{A}\Box \omega)+\overline{A}(\textbf{D}^\alpha \omega\textbf{D}_\alpha A+\frac{1}{2}A\Box \omega)=0.
    \end{align}
    Moreover, $\textbf{U}^\alpha$ satisfies the equation \eqref{eq:eikonalnormaliz} by definition and is timelike because $V'(\rho)$ is positive ($\rho$ is positive by definition and $V'(x)=x^{\gamma-1}$).
\end{proof}
\begin{rem}
\label{rem:jetrhoWKB}
   We see that the momentum $\textbf{J}^\varepsilon$ and the density $\rho^\varepsilon$ of $\Phi_1^\varepsilon$ from Proposition \ref{propal:approxWKB}, are such that
    \begin{align}
        &\textbf{J}^\varepsilon=\frac{i}{2}(\Phi_1^\varepsilon\overline{\textbf{D}\Phi_1^\varepsilon}-\overline{\Phi_1^\varepsilon}\textbf{D}\Phi_1^\varepsilon)=\textbf{D} \omega|A|^2+O(\varepsilon)=\textbf{U}\rho+O(\varepsilon),\\
        &\rho^\varepsilon=|\Phi_1^\varepsilon|^2=|A|^2=\rho,
    \end{align}
    so that $\textbf{U}\rho$ and $\rho$ give a good approximation of them when $\varepsilon$ is small.
\end{rem}

\subsection{Main results}
\label{subsection:mainresults}
In this section, we give the two main Theorems. The first Theorem states that the semi-classical limit of KG \eqref{eq:KGPintro} is the REP system \eqref{eq:EulerREPintro} for a certain class of initial data. Then, the second Theorem states that the fluid limit is in fact the relativistic barotropic Euler system \eqref{eq:relatEulintro} if we look at the right quantities. We start by giving the definition of well-prepared initial data for the REP system \eqref{eq:EulerREPintro}.\\
\begin{defi}
    \label{defi:wellprepREP}
    Let $(\mathscr{U}^0,\mathscr{U},\varrho)$ be initial data for \eqref{eq:EulerREPintro}, i.e. $(\textbf{U}^0,\textbf{U},\rho)|_{t=0}=(\mathscr{U}^0,\mathscr{U},\varrho)$. The initial data are well-prepared if 
    \begin{align}
        &\varrho\geq0,\;\; \mathscr{U}^0>0 &-\mathscr{U}^0\mathscr{U}_0+\mathscr{U}^i\mathscr{U}_i+2V'(\varrho)=-1, 
    \end{align}
    and if $(\sqrt{\varrho},\mathscr{U}^0,\mathscr{U})\in (H^4)^3$. Moreover, if $\gamma\in(2,3)\cup(3,4)\cup(4,5)$ the quantity $\sqrt{\varrho}^{\gamma-1}$ needs to be in $H^4$.\\
\end{defi}

\begin{unTheorem}
\label{unTheorem:TH1mainth}
    Let $(\varphi^\varepsilon,\dot{\varphi}^\varepsilon)_{0<\varepsilon<1}$ be a family of initial data for \eqref{eq:KGPKGP} and let $(\mathscr{U}^0,\mathscr{U},\varrho)$ be initial data for \eqref{eq:EulerbaseREP}. We assume that :
\begin{enumerate}
\label{enumerate:listmainth}
    \item \label{item:item1mainth}The initial data have the regularity $(\varphi^\varepsilon,\dot{\varphi}^\varepsilon)_{0<\varepsilon<1}\in H^{2}\times H^{1}$ and $(\mathscr{U}^0,\mathscr{U},\varrho)\in(H^{4})^3$. 
     \item \label{item:item2mainth}There exists a constant $c_0>0$ such that 
     \begin{align}
         \int_{\mathbb{R}^3}{\frac{\varepsilon^2|\nabla\varphi^\varepsilon|^2}{2}+\frac{\varepsilon^2|\dot{\varphi}^\varepsilon|^2}{2}+\frac{|\varphi^\varepsilon|^2}{2}+V(|\varphi^\varepsilon|^2)dx}+||\mathscr{U}||_{H^4}^2+||\mathscr{U}^0||_{H^4}^2+||\varrho||_{H^4}^2\leq c_0,
     \end{align}
     in particular, the energy of the family $(\varphi^\varepsilon,\dot{\varphi}^\varepsilon)_{0<\varepsilon<1}$ is uniformly bounded.
    \item \label{item:item3mainth}The initial data $(\mathscr{U}^0,\mathscr{U},\varrho)$ are well-prepared, see definition \ref{defi:wellprepREP}.
    \item \label{item:item4mainth}At $t=0$, the modulated energy $H^\varepsilon_0$ of \eqref{eq:moduH0} satisfies 
    \begin{align}     
\label{eq:H0smallmainth}
H^\varepsilon_0(\varphi^\varepsilon,\dot{\varphi}^\varepsilon,\mathscr{U}^0,\mathscr{U},\varrho)=O(\varepsilon^2).
    \end{align}
\end{enumerate}
Then, 
\begin{enumerate}[label=\roman*)]
 \label{itemize:pointsmainth}
	\item  \label{item:point1mainth} There exist $t_\varepsilon<\infty$ and  $\Phi^\varepsilon\in\bigcap^{2}_{j=0}C^{2-j}([0,t_\varepsilon],H^{j})$ for any $\varepsilon\in(0,1)$, such that $(\Phi^\varepsilon)_{0<\varepsilon<1}$ is a family of solutions to \eqref{eq:KGPKGP} with $(\Phi^\varepsilon,\partial_t\Phi^\varepsilon)|_{t=0}=(\varphi^\varepsilon,\dot{\varphi}^\varepsilon)$, and $(\textbf{U},\rho)\in(\bigcap^{4}_{j=0}C^{4-j}([0,T],H^{j}))^2$ a solution to \eqref{eq:EulerbaseREP} for some $T<+\infty$ with $(\textbf{U}^0,\textbf{U},\rho)|_{t=0}=(\mathscr{U}^0,\mathscr{U},\varrho)$. For $2\leq\gamma\leq3$, we can pick $t_\varepsilon=T$.
    \item \label{item:point2mainth}  For all $t\in[0,\min(t_\varepsilon,T)]$ we have
    \begin{align}      
     H^\varepsilon_0(\Phi^\varepsilon,\partial_t\Phi^\varepsilon,\textbf{U},\rho)(t)=O(\varepsilon^2),
    \end{align}
and so for some $\delta>0$
\begin{align}
\label{eq:conv1mainth} 
    ||\textbf{J}^\varepsilon(t)-\textbf{U}\rho(t)||_{L^{\frac{2\gamma}{\gamma+1}}}+||\textbf{J}^\varepsilon(t)-\textbf{U}\rho(t)||_{L^1}+||\rho^\varepsilon(t)-\rho(t)||_{L^{\gamma}}+||V(\rho^\varepsilon)(t)-V(\rho)(t)||_{L^1}=O(\varepsilon^{\delta}).
\end{align}    
    \item \label{item:point3mainth} 
      For $2\leq\gamma\leq3$, we have the strong statement
      \begin{align}    
      \label{eq:conv2mainth} 
     & \lim_{\varepsilon\to 0}||\textbf{J}^\varepsilon-\textbf{U}\rho||_{L^\infty([0,T],L^{\frac{2\gamma}{\gamma+1}})}+||\textbf{J}^\varepsilon-\textbf{U}\rho||_{L^\infty([0,T],L^1)}=0,\\
     &\lim_{\varepsilon\to 0}||\rho^\varepsilon-\rho||_{L^\infty([0,T],L^{\gamma})}+||V(\rho^\varepsilon)-V(\rho)||_{L^\infty([0,T],L^{1})}=0.
    \end{align}
\end{enumerate}
\end{unTheorem}
\begin{unTheorem}
\label{unTheorem:TH2mainth} 
     Let $(\textbf{u},\mu,p,\Gamma)=(\textbf{u}(\textbf{U},\rho,V'(\rho)),\mu(\textbf{U},\rho,V'(\rho)),p(\textbf{U},\rho,V'(\rho)),\Gamma(\textbf{U},\rho,V'(\rho)))$ be given by the change of unknown of Proposition \ref{propal:equivrelatEul} with respect to $(\textbf{U},\rho)$, the solution to \eqref{eq:EulerREPintro} given in Theorem \ref{unTheorem:TH1mainth}. Then, $(\textbf{u},\mu,p)$ corresponds to the four-velocity, the energy density and the pressure solution to the relativistic barotropic Euler \eqref{eq:relatEulintro} and $\Gamma$ to a scaling factor such that, under the assumptions of Theorem \ref{unTheorem:TH1mainth},
    \begin{enumerate}[label=\roman*)]
    \label{itemize:pointsmainth2}
            \item \label{item:point4mainth}For all $t\in[0,\min(t_\varepsilon,T)]$ and some $\delta>0$, we have
\begin{align}
\label{eq:conv3mainth}
    &||\textbf{J}^\varepsilon(t)-\Gamma\textbf{u}(\mu+p)(t)||_{L^{\frac{2\gamma}{\gamma+1}}}+||\textbf{J}^\varepsilon(t)-\Gamma\textbf{u}(\mu+p)(t)||_{L^1}=O(\varepsilon^{\delta}),\\
    &||\rho^\varepsilon(t)-\Gamma^2(\mu+p)(t)||_{L^{\gamma}}+||(\gamma-1)V(\rho^\varepsilon)(t)-p(t)||_{L^{1}}=O(\varepsilon^{\delta}).
\end{align}    
    \item \label{item:point5mainth} For $2\leq\gamma\leq3$, we have the strong statement
      \begin{align}     
      \label{eq:conv4mainth}
      &\lim_{\varepsilon\to 0}||\textbf{J}^\varepsilon-\Gamma\textbf{u}(\mu+p)||_{L^\infty([0,T],L^{\frac{2\gamma}{\gamma+1}})}+||\textbf{J}^\varepsilon-\Gamma\textbf{u}(\mu+p)||_{L^\infty([0,T],L^1)}=0,\\
      &\lim_{\varepsilon\to 0}||\rho^\varepsilon-\Gamma^2(\mu+p)||_{L^\infty([0,T],L^{\gamma})}+||V(\rho^\varepsilon)-p||_{L^\infty([0,T],L^{1})}=0.
    \end{align}
    \end{enumerate}
    
\end{unTheorem}
The assumption \ref{item:item1mainth}  is present for obvious technical purposes. The second \ref{item:item2mainth} is not restrictive in the sense that the typical solutions we are interested in satisfy it. Indeed, even explicit high frequency ansatz\footnote{See for example the ansatz for the approximate solution of section \ref{subsection:WKB}.}, for which the derivatives behave badly, have energies that are uniformly bounded in $\varepsilon$. This comes from the presence of $\varepsilon$ in the definition of the energy (given in \eqref{eq:EenergyKGP} in section \ref{section:KGPsec}) to balance the size of the derivatives. The assumptions \ref{item:item3mainth} (which refers to Definition \ref{defi:wellprepREP}) and \ref{item:item4mainth} are more generic than they seem.
Typically, one can think of exact solutions $\Phi^\varepsilon$ to \eqref{eq:KGPintro} as based on mono phase high frequency approximate\footnote{We recall that we only use the WKB method heuristically here.} solution 
\begin{equation}
\label{eq:WKBtypeinitdata}
    \Phi_1^\varepsilon=e^{i\frac{\omega}{\varepsilon}}A
\end{equation}
as in section \ref{subsection:WKB}, and with initial data $(\varphi_1^\varepsilon,\dot{\varphi}_1^\varepsilon)_{0<\varepsilon<1}=(a e^{\frac{v}{\varepsilon}},\dot{a} e^{\frac{v}{\varepsilon}}+i\frac{\dot{v}}{\varepsilon}a^\varepsilon e^{\frac{v}{\varepsilon}})$. In particular, the initial data for the phase $(\omega,\partial_t\omega)|_{t=0}=(v,\dot{v})$ must be solutions to the eikonal equation \eqref{eq:eikonalWKB}. Then, one can pick $(\mathscr{U}_0,\mathscr{U})=(\dot{v},\nabla v)$ and $\varrho=|a|^2$, this implies the normalization for $(\mathscr{U}_0,\mathscr{U},\varrho)$ and the fact that $\varrho\geq0$ in Definition \ref{defi:wellprepREP}. The ansatz \eqref{eq:WKBtypeinitdata} gives natural profiles $(\mathscr{U}^0,\mathscr{U},\varrho)$ that fit the criteria \ref{item:item3mainth} and \ref{item:item4mainth}\footnote{If they have the good regularity, see below.}. In other words, the Theorem typically deals with mono phase high frequency WKB ansatz. \\
Moreover, we observe that even if $v=0$ ($\varphi_1^\varepsilon$ is not initially highly oscillating) then (due to the eikonal equation \eqref{eq:eikonalWKB}) $\dot{v}$ is not zero. The solution becomes instantaneously highly oscillatory. This remark is made in \cite{zbMATH05243173} on the same type of initial data for the semi-classical Schrödinger equation. \\
More generally, we see that the solutions "have to be" highly oscillating in at least one direction because of the scaling of the equation, the derivatives of $\Phi^\varepsilon$ have to be large in comparison with its amplitude. The assumption \ref{item:item4mainth} (and more generally the modulated energy approach) morally means that there is locally only one direction in which the solution is high frequency. This is captured by the vector field $\textbf{U}$.\\
Finally, we require some regularity for $\sqrt{\varrho}$ to have well-prepared initial data. This restriction makes sense because $\rho$ is thought to be an approximation of $\rho^\varepsilon=|\Phi^\varepsilon|^2$ from which we expect that $\sqrt{\rho}$ has the regularity of $|\Phi^\varepsilon|$. This and the other assumptions of Definition \ref{defi:wellprepREP} are mandatory for the proof of the well-posedness of the REP system in Proposition \ref{propal:wellposedEulerEP}. \\
It is already interesting to note that the properties of the well-prepared $(\mathscr{U}^0,\mathscr{U},\varrho)$ propagate, see Proposition \ref{propal:propagwellprepREP}. 
In particular, we take advantage of the eikonal equation (the normalization of the vector field $\textbf{U}$ in Proposition \ref{propal:propagwellprepREP}) during the proof of Proposition \ref{propal:wellposedEulerEP} for the well posedness of the REP system and in section \ref{subsection:Pgproof} to prove the propagation of smallness for the modulated energy. We note also that this particular normalization is a direct equivalent of the $-1$ normalization for the usual relativistic Euler, see section \ref{section:relatEul}.  \\ \\
The second Theorem is a direct consequence of the first and the equivalence of the REP system and the relativistic Euler equations given by Proposition \ref{propal:equivrelatEul}.
\subsection{Comments}
\label{subsection:comments}
Overall, we prove that the semiclassical (monokinetic) limit of the massive Klein-Gordon equation with a defocusing power-law potential is the REP system, a relativistic Euler system with a potential. The most interesting cases, for which we have the convergence on an interval independent of $\varepsilon$, are the subcritical and critical ones, when $2\leq\gamma\leq3$.  Moreover, the REP system is equivalent to the usual relativistic barotropic Euler equations (see section \ref{section:relatEul}). This implies that the semiclassical (monokinetic) limit of massive KG with a defocusing power-law potential is the relativistic Euler equations up to a rescaling. This is summed up in the following picture \\
\begin{center}
\begin{tikzpicture}[
roundnode/.style={circle, draw=black!60, fill=black!5, very thick, minimum size=7mm},
squarednode/.style={rectangle, draw=black!60, fill=black!5, very thick, minimum size=5mm},
]
\node[squarednode]      (sch)   at (3,3)                       {defocusing Schrödinger};
\node[squarednode]      (kg)        [below=of sch]                       {defocusing massive Klein-Gordon};
\node[squarednode]      (relat)     [right=of kg]                        {relativistic barotropic Euler};
\node[squarednode]      (euler)          [above=of relat]                     {compressible Euler};

\draw[->] (kg.north) -- (sch.south) node[midway,left] {(1)};
\draw[->] (kg.east) -- (relat.west)node[midway,above] {(5)};
\draw[->] (sch.east) -- (euler.west)node[midway,above] {(2)};
\draw[->] (relat.north) -- (euler.south)node[midway,right] {(4)};
\draw[->] (kg.north) -- (euler.south)node[midway,above] {(3)};
\draw[thick,->] (0,0) -- (4.5,0)node[midway,below] {semi-classical limit};
\draw[thick,->] (0,0) -- (0,4.5)node[midway,above,sloped]{non-relativistic limit};;
\end{tikzpicture}
\end{center}
\begin{itemize}
   \item[--]The point (1) is proved in \cite{zbMATH01837395}, \cite{zbMATH01747314} for any dimension greater than 2, and also in various settings in \cite{zbMATH03871951}, \cite{NAJMAN1990217} \cite{zbMATH06303453}, \cite{zbMATH01650142}
   \cite{pasquali2018dynamicsnonlinearkleingordonequation}, \cite{lei2023nonrelativistic} and \cite{bambusi2025nonrelativisticlimitnonlinear}. 
   \item [--]The point (2) is proved in \cite{zbMATH05243173} via the modulated energy method and in \cite{zbMATH01132286} via the WKB method. One can also look at \cite{zbMATH05000052} for an application of modulated energy method on the study of the semi-classical limit of Gross-Pitaevskii, a very related equation. 
   \item [--]The point (3) is proved in \cite{zbMATH06101438} and \cite{zbMATH06039446} via the modulated energy method, the double limit is performed as a simple limit by choosing $c^{-1}=\varepsilon^\kappa$ for some $\kappa>0$.
   \item [--]The point (4) is proved in \cite{zbMATH07119817}, \cite{zbMATH07531778}, \cite{MAI2022336}, \cite{zbMATH05077591}, \cite{zbMATH00840326}, \cite{zbMATH01309659} and \cite{zbMATH07114642} in dimension 1 using various techniques.
   \item [--]The point (5) is proved in the present paper.
\end{itemize}
The main constraints of this result are that:
\begin{itemize}
    \item We only deal with the massive case (see for example this requirement in Proposition \ref{propal:c2c3Cyproof}).
    \item The proof requires a potential (the convergence of the density $\rho^\varepsilon$ is obtained with the potential, see Remark \ref{rem:convdensityneedCyproof}). 
    \item The solution to the REP system is supposed to be regular enough ($\textbf{U}$ and $\rho$ in $H^4$). For example, in \cite{zbMATH01442880} (where the modulated energy method is introduced), the solution to the fluid system is understood in the sense of a dissipative solution (defined in \cite{zbMATH00928933}), with very low regularity. 
    \item We only deal with one direction of oscillation (we see that there is only one vector field $\textbf{U}$).\\
\end{itemize}
\begin{rem} 
\label{rem:instabil} For the last point, we see that the modulated energy method does not say anything about cases where there is more than one direction of high oscillation. In the WKB language, this translates to : the modulated energy method does not deal with multi-phase ansatz. We refer to \cite{METIVIER2009169} and \cite{zbMATH06039454} for general results on multi-phase high-frequency solutions. We conjecture that the WKB multi-phase superposition is unstable for the nonlinear Klein-Gordon equation, as it is for the nonlinear Schrödinger equation. We refer to \cite{carles2023nonlineareffectsmultiphasewkb}, \cite{zbMATH05243173} et \cite{Han_Kwan_2016} on the latter subject.\\
\end{rem}
\begin{rem}
    To recover a non-monokinetic Vlasov system at the semi-classical limit, the approach can be more direct than the WKB method if we pass to the mixed state regime. For example, with the study of the Wigner transform. We refer to \cite{zbMATH00482230}, \cite{zbMATH00203353}, for works on the Schrödinger-Poisson equations. In \cite{Mauser2002SemiclassicalLO} and \cite{Carles_2009}, limitations are presented for the application of this method to the nonlinear defocusing Schrödinger equation. These limitations are linked to the instability mentioned in the previous Remark \ref{rem:instabil}. No works of this type have been done on the semi-classical Klein-Gordon equation as far as we know.\\
\end{rem}
\begin{rem}
We think that our work might be done for other geometries of spacetime, the arguments given in the present paper are not restricted to the Minkowski metric.
\end{rem}
\subsection{Idea of the proof}
\label{subsection:ideaofproof}
We only need to show the first Theorem \ref{unTheorem:TH1mainth} as the second one is a corollary once we have the right change of variable, see Proposition \ref{propal:equivrelatEul}.
Firstly, we need to show the existence of sufficiently regular solutions to \eqref{eq:KGPintro} and \eqref{eq:EulerREPintro} with initial data $(\varphi^\varepsilon,\dot{\varphi}^\varepsilon)_{0<\varepsilon<1}$ and $(\mathscr{U}^0,\mathscr{U},\varrho)$.\\
For the semi-classical Klein-Gordon equation \eqref{eq:KGPintro}, it is easy to prove local well-posedness with energy methods, but we need global well-posedness to have a uniform time of existence for the family of solutions $(\Phi^\varepsilon)_{0<\varepsilon<1}$ and to recover the strong statement \ref{item:point3mainth} of the Theorem \ref{unTheorem:TH1mainth}. The global well-posedness for $2\leq\gamma\leq3$ is a common result given in \cite{Brenner1979,Wahl1981,zbMATH03875887} for the subcritical case and in \cite{zbMATH04202876,zbMATH00011219,zbMATH00168350,zbMATH00703984,zbMATH00817689} for the critical one.
For the REP system, we show in Appendix \ref{section:appendix2} that it is a symmetrizable hyperbolic system (the proof is inspired by \cite{zbMATH04039498}) and so it is locally well-posed for sufficiently regular initial data. We add that $(\textbf{U},\rho)$ must be more regular than what $\Phi^\varepsilon$ needs to be to perform the operation in section \ref{subsection:Cyproof} and section \ref{subsection:Pgproof}. \\
Then, the rest of the proof concerns the convergence of the momentum and the density. The core of the argument is to construct a quantity $H^\varepsilon$ that
\begin{enumerate}[label=\arabic*)]
\label{enumerate:listmodpointIdea}
    \item \label{item:modpoint1Idea} Controls the convergence in the sense that 
    \begin{align}       &H^\varepsilon(t)=O(\varepsilon^2)\implies\\
    &||\textbf{J}^\varepsilon(t)-\textbf{U}\rho(t)||_{L^{\frac{2\gamma}{\gamma+1}}}+||\textbf{J}^\varepsilon(t)-\textbf{U}\rho(t)||_{L^1}+||\rho^\varepsilon(t)-\rho(t)||_{L^\gamma}+||V(\rho^\varepsilon)(t)-V(\rho)(t)||_{L^1}=O(\varepsilon^{\delta})\nonumber,
    \end{align}
    for some $\delta>0$. This is the \textbf{coercivity property}.
    \item \label{item:modpoint2Idea} Propagates the smallness in the sense that 
    \begin{align}       
    H^\varepsilon(0)=O(\varepsilon^2)\implies H^\varepsilon(t)=O(\varepsilon^2),
    \end{align}
    for $t$ finite such that the solutions exist and remain regular. This is the \textbf{propagation property}.
\end{enumerate}
We can directly construct an equivalent of the modulated energy of \cite{zbMATH05243173}, made for the semi-classical limit of Schrödinger, with $H^\varepsilon_0$ defined as in \eqref{eq:moduH0}. We show in Proposition \ref{propal:tryoneh00Cyproof} that $H^\varepsilon_0$ actually satisfies point \ref{item:modpoint1Idea}. Nonetheless, this first candidate eventually fails to satisfy point \ref{item:modpoint2Idea} with the techniques used here\footnote{In fact, we can show that it satisfies point \ref{item:modpoint2Idea} a posetriori, once we show that the new $H^\varepsilon$ satisfies it first.}. \\
To fix that, we build a new "modulated energy" $H^\varepsilon$ based on the full stress-energy tensor $T^\varepsilon$ of \eqref{eq:stressenergintro} (and not just on $T^\varepsilon_{00}$ which corresponds to the energy in the considered frame) that is equivalent to the previous one in terms of control and coercivity, $H^\varepsilon\sim H^\varepsilon_0$, but with a better structure for point \ref{item:modpoint2Idea}. This is proved in Proposition \ref{propal:HgoodfinalCyproof}.\\
The proof follows as the ones of \cite{zbMATH05243173} or \cite{zbMATH06101438}, that is, to show that 
\begin{equation}
\label{eq:modderivIdea}
    \frac{d}{dt}H^\varepsilon\leq H^\varepsilon+O(\varepsilon^2).
\end{equation}
Another novelty is that we have to use the eikonal equation (the unusual normalization of the four-velocity $\textbf{U}$) to show \eqref{eq:modderivIdea}. We also use the fact that the potential is a power law to get the inequality \eqref{eq:modderivIdea} in a straightforward way.\\
The Lebesgue spaces in which the convergences occur are directly linked to the degree $\gamma-1$ of the power law $V'$. It is logical that the convergence of the density takes place in $L^\gamma$ (the norm of the potential energy density) but we can observe that the convergence of the momentum depends on the convergence of the density. In fact, if there is no potential, we cannot prove the convergence of the density, nor the momentum, see Remark \ref{rem:convdensityneedCyproof}.\\
\subsection{Acknowledgments}
The author thanks Cécile Huneau for her precious help during the elaboration of this paper, Daniel Han-Kwan for his important remarks, and the anonymous reviewers for their great recommendations. 
\subsection{Outline of the paper}
\label{subsection:outline}
\begin{itemize}
    \item In section \ref{section:KGPsec} we give all the details on the semi-classical Klein-Gordon equation.
    \item  In section \ref{section:REP} we give all the details on the REP system.
    \item In section \ref{section:defmod} we give the basic definitions for the new modulated energy.
    \item In section \ref{section:Proof} we give the proof of Theorem \ref{unTheorem:TH1mainth}. 
    \item In section \ref{section:relatEul} we show the equivalence between the REP system and the relativistic barotropic Euler system.
    \item In the Appendix \ref{section:appendix2}, we give the well-posedness result for the REP system.
\end{itemize}
\section{Notations}
\label{section:nota}

\begin{nota}
\label{nota:gradNOTA1}
    The symbol $\textbf{D}$ alone denotes the spacetime gradient and we denote $|\textbf{D} f|^2=\textbf{D}f\cdot\overline{\textbf{D}f}=\partial_tf\overline{\partial_tf}+\nabla f\cdot\overline{\nabla f}$ with $\nabla$ the space gradient and $\cdot$ the usual scalar products of both $\mathbb{R}^3$ and $\mathbb{R}^4$.\\
\end{nota}
\begin{nota}
\label{nota:gradNOTA2}
    We denote with bold letters the four-vectors, such as $\textbf{X}$, and with roman letters, $X$ with $X^i=\textbf{X}^i$, the three-dimensional vectors made of its space components. We note $X^0=\textbf{X}^0$ its time component.
\end{nota}
\section{Klein-Gordon equation}
\label{section:KGPsec}
We give the details on the semi-classical version of the massive Klein-Gordon equation with a potential
\begin{equation}
\label{eq:KGPKGP}	\varepsilon^2\Box\Phi^\varepsilon=\Phi^\varepsilon+2V'(|\Phi^\varepsilon|^2)\Phi^\varepsilon.
\end{equation}
\begin{defi}
\label{defi:energytensorKGP}
The stress-energy tensor associated to \eqref{eq:KGPKGP} is 
\begin{equation}
\label{eq:stressenergyKGP}
    T^{KG}[\Phi^\varepsilon]_{\alpha\beta}=\frac{\varepsilon^2}{2}(\textbf{D}_\alpha\Phi^\varepsilon\textbf{D}_\beta\overline{\Phi^\varepsilon}+\textbf{D}_\alpha\overline{\Phi^\varepsilon}\textbf{D}_\beta\Phi^\varepsilon)-\frac{1}{2}g_{\alpha\beta}(\varepsilon^2\textbf{D}_\mu\overline{\Phi^\varepsilon}\textbf{D}^\mu\Phi^\varepsilon+|\Phi^\varepsilon|^2+2V(|\Phi^\varepsilon|^2)),
\end{equation}
its divergence is equal to 0 
\begin{equation}
\label{eq:divstressenergyKGP}
    \textbf{D}_\alpha (T^{KG})^{\alpha}~_\beta=0.
\end{equation}
We also have
\begin{equation}
\label{eq:EstressenergyKGP}
   \int{T^{KG}[\Phi^\varepsilon]_{00}dx}=\mathcal{E}^{KG}[\Phi^\varepsilon],
\end{equation}
where $\mathcal{E}^{KG}$ is the associated conserved energy
\begin{equation}
\label{eq:EenergyKGP}	
	\mathcal{E}^{KG}[\Phi^\varepsilon]=\int_{\mathbb{R}^3}{\frac{\varepsilon^2|\textbf{D}\Phi^\varepsilon|^2}{2}+\frac{|\Phi^\varepsilon|^2}{2}+V(|\Phi^\varepsilon|^2)dx}.
\end{equation}
\end{defi}
\begin{defi}
\label{defi:momentumKGP}
    The momentum (or charge current density) associated to \eqref{eq:KGPKGP} is
\begin{equation}
\label{eq:momentumequationKGP}
\textbf{J}^\varepsilon_\alpha=-\Im(\Phi^\varepsilon\overline{\varepsilon \textbf{D}_\alpha\Phi^\varepsilon}).
\end{equation}
We also set $J^\varepsilon$ from $\mathbb{R}^{1+3}$ to $\mathbb{R}^{3}$ as  $J_i^\varepsilon=\textbf{J}^\varepsilon_i$ and $J_0^\varepsilon=\textbf{J}^\varepsilon_0$. We have the equation 
\begin{equation}
\label{eq:momentumJ0KGP}
\partial_tJ_0^\varepsilon=\nabla\cdot J^\varepsilon,
\end{equation}
written equivalently as 
\begin{equation}
\label{eq:momentumbisKGP}
\textbf{D}_\alpha (\textbf{J}^\varepsilon)^\alpha=0,
\end{equation}
and
\begin{equation}
\label{eq:currentJKGP}
\partial_tJ^\varepsilon= -\Im(\partial_t\Phi^\varepsilon\overline{\varepsilon\nabla\Phi^\varepsilon})+\nabla J_0^\varepsilon.
\end{equation}
\end{defi}
\begin{nota}
\label{nota:densityKGP}
To simplify the notation we set the density $\rho^\varepsilon=|\Phi^\varepsilon|^2$.\\
\end{nota}
\begin{rem}
\label{rem:chargeKGP}	
We see that the density $\rho^\varepsilon$ and the charge density (the component 0 of the momentum four-vector) $(J^\varepsilon)^0$ are different quantities. In the non-relativistic case (for the Schrödinger equation) they are the same quantity, the probability density. We say that $(J^\varepsilon)^0$ is a charge density because it is not positive definite but its total charge is conserved as the total probability for Schrödinger. We have 
\begin{equation}
\label{eq:conservedchargeKGP}
	\frac{d}{dt}\int_{\mathbb{R}^3}{(J^\varepsilon)^0 dx}=0.
\end{equation}
Moreover, we have  
\begin{align}
	&\frac{d}{dt}<x>=\frac{d}{dt}\int_{\mathbb{R}^3}{(J^\varepsilon)^0 x dx}=\int_{\mathbb{R}^3}{J^\varepsilon  dx},
\end{align}
which is an analogous of the Ehrenfest Theorem applied to the position operator. 
\end{rem}
We add the following useful equations to our list.\\
\begin{propal}
    \label{propal:splitequationKGP}
   Let $\Phi^\varepsilon$ be a wave function, then we have 
    \begin{equation}
       \label{eq:splitJJ1KGP}
       \frac{(\textbf{J}^\varepsilon)^\alpha(\textbf{J}^\varepsilon)_\alpha}{\rho^\varepsilon}=-\varepsilon^2\textbf{D}_\alpha\sqrt{\rho^\varepsilon}\textbf{D}^\alpha\sqrt{\rho^\varepsilon}+\varepsilon^2\textbf{D}_\alpha\Phi^\varepsilon\overline{\textbf{D}^\alpha\Phi^\varepsilon},
    \end{equation}
     \begin{equation}
         \label{eq:splitJJ2KGP}
         \frac{\textbf{J}^\varepsilon\cdot\textbf{J}^\varepsilon}{\rho^\varepsilon}=-\varepsilon^2\textbf{D}\sqrt{\rho^\varepsilon}\cdot\textbf{D}\sqrt{\rho^\varepsilon}+\varepsilon^2|\textbf{D}\Phi^\varepsilon|^2,
    \end{equation}
   and if $\Phi^\varepsilon$ is a solution to \eqref{eq:KGPKGP}, then the momentum $\textbf{J}^\varepsilon$ and the density $\rho^\varepsilon$ are solutions to
    \begin{equation}
         \label{eq:splitJJ3KGP}
         \frac{(\textbf{J}^\varepsilon)^\alpha(\textbf{J}^\varepsilon)_\alpha}{\rho^\varepsilon}=\varepsilon^2\sqrt{\rho^\varepsilon}\Box\sqrt{\rho^\varepsilon}-\rho^\varepsilon-2V'(\rho^\varepsilon)\rho^\varepsilon.
            \end{equation}
\end{propal}
\begin{proof}
    By direct calculations, we get  
               \begin{align*}
         \frac{(\textbf{J}^\varepsilon)^\alpha(\textbf{J}^\varepsilon)_\alpha}{|\Phi^\varepsilon|^2}&=-\frac{\varepsilon^2(\Phi^\varepsilon\overline{\textbf{D}_\alpha\Phi^\varepsilon}\Phi^\varepsilon\overline{\textbf{D}^\alpha\Phi^\varepsilon}+\overline{\Phi^\varepsilon}\textbf{D}_\alpha\Phi^\varepsilon\overline{\Phi^\varepsilon}\textbf{D}^\alpha\Phi^\varepsilon-2\textbf{D}^\alpha\Phi^\varepsilon\overline{\textbf{D}_\alpha\Phi^\varepsilon}|\Phi^\varepsilon|^2)}{4|\Phi^\varepsilon|^2}\\
     &=-\varepsilon^2(\frac{\textbf{D}_\alpha|\Phi^\varepsilon|\Phi^\varepsilon\overline{\textbf{D}^\alpha\Phi^\varepsilon}}{2|\Phi^\varepsilon|}-\frac{\overline{\Phi^\varepsilon}\textbf{D}^\alpha\Phi^\varepsilon\Phi^\varepsilon\overline{\textbf{D}_\alpha\Phi^\varepsilon}}{4|\Phi^\varepsilon|^2}+\frac{\textbf{D}_\alpha|\Phi^\varepsilon|\overline{\Phi^\varepsilon}\textbf{D}^\alpha\Phi^\varepsilon}{2|\Phi^\varepsilon|}-\frac{\overline{\Phi^\varepsilon}\textbf{D}^\alpha\Phi^\varepsilon\Phi^\varepsilon\overline{\textbf{D}_\alpha\Phi^\varepsilon}}{4|\Phi^\varepsilon|^2})\\
     &+\varepsilon^2\frac{2\textbf{D}^\alpha\Phi^\varepsilon\overline{\textbf{D}_\alpha\Phi^\varepsilon}|\Phi^\varepsilon|^2}{4|\Phi^\varepsilon|^2}\\
     &=-\varepsilon^2\textbf{D}_\alpha|\Phi^\varepsilon|\textbf{D}^\alpha|\Phi^\varepsilon|+\varepsilon^2\textbf{D}_\alpha\Phi^\varepsilon\overline{\textbf{D}^\alpha\Phi^\varepsilon},
    \end{align*}
    for the first equality. The second is obtained with the same techniques. Then, we have
    \begin{align*}
         \frac{(\textbf{J}^\varepsilon)^\alpha(\textbf{J}^\varepsilon)_\alpha}{|\Phi^\varepsilon|^2}&=-\varepsilon^2\textbf{D}^\alpha(\textbf{D}_\alpha|\Phi^\varepsilon||\Phi^\varepsilon|)+\varepsilon^2\Box|\Phi^\varepsilon||\Phi^\varepsilon|+\varepsilon^2\frac{\textbf{D}^\alpha(\textbf{D}_\alpha\Phi^\varepsilon\overline{\Phi^\varepsilon})+\textbf{D}^\alpha(\Phi^\varepsilon\textbf{D}_\alpha\overline{\Phi^\varepsilon})}{2}-\varepsilon^2\frac{\Box\Phi^\varepsilon\overline{\Phi^\varepsilon}+\Phi^\varepsilon\Box\overline{\Phi^\varepsilon}}{2}\\
         &=-\varepsilon^2\frac{\textbf{D}^\alpha(\Phi^\varepsilon\textbf{D}_\alpha\overline{\Phi^\varepsilon}+\Phi^\varepsilon\textbf{D}_\alpha\overline{\Phi^\varepsilon})}{2}+\varepsilon^2\Box|\Phi^\varepsilon||\Phi^\varepsilon|+\varepsilon^2\frac{\textbf{D}^\alpha(\textbf{D}_\alpha\Phi^\varepsilon\overline{\Phi^\varepsilon})+\textbf{D}^\alpha(\Phi^\varepsilon\textbf{D}_\alpha\overline{\Phi^\varepsilon})}{2}\\
          &-\varepsilon^2\frac{\Box\Phi^\varepsilon\overline{\Phi^\varepsilon}+\Phi^\varepsilon\Box\overline{\Phi^\varepsilon}}{2}\\
          &=\varepsilon^2\sqrt{\rho^\varepsilon}\Box\sqrt{\rho^\varepsilon}-\varepsilon^2\frac{\Box\Phi^\varepsilon\overline{\Phi^\varepsilon}+\Phi^\varepsilon\Box\overline{\Phi^\varepsilon}}{2},
    \end{align*}
    so that if $\Phi^\varepsilon$ is a solution to \eqref{eq:KGPKGP} we recover
     \begin{align*}
         &\frac{(\textbf{J}^\varepsilon)^\alpha(\textbf{J}^\varepsilon)_\alpha}{|\Phi^\varepsilon|^2}=\varepsilon^2\sqrt{\rho^\varepsilon}\Box\sqrt{\rho^\varepsilon}-\rho^\varepsilon-2V'(\rho^\varepsilon)\rho^\varepsilon.
    \end{align*}
\end{proof}
In the next Proposition, we state the global well-posedness result.\\
\begin{propal}[From \cite{Brenner1979,Wahl1981,zbMATH03875887,zbMATH04202876,zbMATH00011219,zbMATH00168350,zbMATH00703984,zbMATH00817689}]
\label{propal:solutglobalKGP}
    Let $(\varphi^\varepsilon,\dot{\varphi}^\varepsilon)$ be in $H^2\times H^1$ with
      \begin{align}
         \int_{\mathbb{R}^3}{\frac{\varepsilon^2|\nabla\varphi^\varepsilon|^2}{2}+\frac{\varepsilon^2|\dot{\varphi}^\varepsilon|^2}{2}+\frac{|\varphi^\varepsilon|^2}{2}+V(|\varphi^\varepsilon|^2)dx}\leq C_\varepsilon,
     \end{align}
    then, for $2\leq\gamma\leq3$ there exists a global solution to \eqref{eq:KGPKGP} with initial data $(\Phi^\varepsilon,\partial_t\Phi^\varepsilon)|_{t=0}=(\varphi^\varepsilon,\dot{\varphi}^\varepsilon)$ and with $\Phi^\varepsilon\in\bigcap^{2}_{j=0}C^{2-j}(\mathbb{R}_+,H^{j})$.\\
\end{propal}
\begin{proof}
     See \cite{Brenner1979,Wahl1981,zbMATH03875887} for the subcritical case $2\leq\gamma<3$ and \cite{zbMATH04202876,zbMATH00011219,zbMATH00168350,zbMATH00703984,zbMATH00817689} for the critical one $\gamma=3$.
\end{proof}
\section{Relativistic Euler with potential}
\label{section:REP}
We introduce the relativistic Euler system with potential in its general form
\begin{equation}
\label{eq:Eulerbasepotent}
\begin{cases}
\textbf{U}^\alpha\textbf{D}_\alpha \textbf{U}_\beta+\dfrac{\textbf{D}_\beta B}{\rho}=0,\\
\textbf{U}^\alpha\textbf{D}_\alpha \rho+\textbf{D}_\alpha \textbf{U}^\alpha\rho=0,\\
B=B(\rho),
\end{cases}
\end{equation}
where $\rho$ is the energy density, $B$ the potential and $\textbf{U}$ the four-velocity vector field. We denote by $U$ the vector field from $\mathbb{R}^{1+3}$ to $\mathbb{R}^{3}$ with $U_i=\textbf{U}_i$, the space components, and by $U_0=\textbf{U}_0$ the time component. \\
\begin{defi}
\label{defi:energytensorREP}
The stress-energy tensor associated to \eqref{eq:EulerbaseREP} is 
\begin{equation}
\label{eq:stressenergyREP}
    T^{EP}[\textbf{U},\rho]_{\alpha\beta}=\rho \textbf{U}_\alpha \textbf{U}_\beta+g_{\alpha\beta}B,
\end{equation}
its divergence is equal to 0 
\begin{equation}
\label{eq:divstressenergyREP}
    \textbf{D}_\alpha (T^{EP})^{\alpha}~_\beta=0.
\end{equation}
We also have
\begin{equation}
\label{eq:EenergystressenergyREP}
   \int{T^{EP}[\textbf{U},\rho]_{00}dx}=\mathcal{E}^{EP}[\textbf{U},\rho],
\end{equation}
where $\mathcal{E}^{EP}$ is the associated conserved energy
\begin{equation}
\label{eq:EenergyREP}	
	\mathcal{E}^{EP}[\textbf{U},\rho]=\int_{\mathbb{R}^3}{\textbf{U}_0\textbf{U}_0\rho-Bdx}.
\end{equation}
\end{defi}
\begin{defi}
\label{defi:momentumREP}
    The momentum associated to \eqref{eq:EulerbaseREP} is
\begin{equation}
\textbf{J}_\alpha=\rho \textbf{U}_\alpha,
\end{equation}
with
\begin{equation}
\label{eq:momentumREP}
\textbf{D}_\alpha \textbf{J}^\alpha=0.
\end{equation}
We also set $J$ from $\mathbb{R}^{1+3}$ to $\mathbb{R}^{3}$ as  $J_i=\textbf{J}_i$ and $J_0=\textbf{J}_0$. 
   We have the conservation law 
\begin{equation}
\label{eq:conservchargeREP}
	\frac{d}{dt}\int_{\mathbb{R}^3}{J_0dx}=\frac{d}{dt}\int_{\mathbb{R}^3}{U_0\rho dx}=0.
\end{equation} 
\end{defi}

\begin{rem}
\label{rem:rhoremarkREP}
The solution $\rho$ to the conservative transport 
\begin{equation}
\label{eq:conservtpREP}
    \textbf{U}^\alpha\textbf{D}_\alpha \rho+\textbf{D}_\alpha \textbf{U}^\alpha\rho=0\\
\end{equation}
is given implicitly by 
\begin{equation}
\label{eq:implicitrhoREP}
    \rho(\chi(\tau,y))=\rho(0,y)e^{-\int^\tau_0{\textbf{D}_\alpha \textbf{U}^\alpha(\chi(\theta,y))d\theta}},
\end{equation}
where $\chi$ is the flow of $\textbf{U}$, solution to 
\begin{equation}
    \begin{cases}
        \label{eq:flowlineREP}\dot{\chi}^\alpha(\tau,y)=\textbf{U}^\alpha(\chi(\tau,y)),\\
        \chi(0,y)=(0,y).
    \end{cases}
\end{equation}
We observe that if $\rho(0)\geq0$ then $\rho(t)\geq0$. Moreover, if $\textbf{U}$ is future-directed ($\textbf{U}^0>0$) then $J^0$ is positive definite. 
\end{rem}
Now, we restrain ourselves to the specific case $B=(\gamma-1)V(\rho)$ with the normalization $\textbf{U}^\alpha \textbf{U}_\alpha+1+2V'(\rho)=0$. In this case, we recover the REP system 
\begin{equation}
\label{eq:EulerbaseREP}
\begin{cases}
\textbf{U}^\alpha\textbf{D}_\alpha \textbf{U}_\beta+\textbf{D}_\beta V'(\rho)=0,\\
\textbf{U}^\alpha\textbf{D}_\alpha \rho+\textbf{D}_\alpha \textbf{U}^\alpha\rho=0,\\
\textbf{U}^\alpha \textbf{U}_\alpha+2V'(\rho)=-1.
\end{cases}
\end{equation}
We compare \eqref{eq:EulerbaseREP} to the usual relativistic barotropic Euler in section \ref{section:relatEul}.\\

\begin{propal}
\label{propal:propagwellprepREP}
    Let $(\textbf{U},\rho)$ be a continuous solution to \eqref{eq:Eulerbasepotent} with $B=(\gamma-1)V(\rho)$ defined on $[0,T^\star]$ with $(\mathscr{U}^0,\mathscr{U},\varrho)$ well-prepared initial data from definition \ref{defi:wellprepREP}.
    Then, for all $t\in[0,T^\star]$
     \begin{align}
     \label{eq:propagwellprepREP}
        &\rho(t)\geq0,\;\; \textbf{U}^0(t)>0, &\textbf{U}^\alpha(t) \textbf{U}_\alpha(t)+1+2V'(\rho(t))=0, 
    \end{align}
    so that $\textbf{U}$ is future-directed and in particular $\textbf{U}^0\geq 1$. Moreover, $(\textbf{U},\rho)$ is a solution to \eqref{eq:EulerbaseREP}.
\end{propal}
\begin{proof}
    We show that the quantity $\textbf{U}^\alpha(t) \textbf{U}_\alpha(t)+1+2V'(\rho(t))$ is constant along the flow lines. We have
    \begin{align*}
         &\textbf{U}^\beta\textbf{D}_\beta(\textbf{U}^\alpha \textbf{U}_\alpha+1+2V'(\rho))=2(\textbf{U}^\alpha \textbf{U}^\beta\textbf{D}_\beta \textbf{U}_\alpha+\textbf{U}^\beta\textbf{D}_\beta V'(\rho))=2( -\textbf{U}^\beta\frac{B'}{\rho}\textbf{D}_\beta \rho +\textbf{U}^\beta V''(\rho)\textbf{D}_\beta\rho )=0.
    \end{align*}
Then, with the formula \eqref{eq:implicitrhoREP} we get that $\rho\geq0$. Finally, we have from \eqref{eq:propagwellprepREP} that 
\begin{align}
\textbf{U}^0=\sqrt{1+2V'(\rho)+\textbf{U}^i\textbf{U}_i}\geq 1.
\end{align}
\end{proof}
\begin{propal}
    \label{propal:wellposedEulerEP}
   The REP system \eqref{eq:EulerbaseREP} is locally well posed for $(\mathscr{U}^0,\mathscr{U},\varrho)$ well-prepared initial data from definition \ref{defi:wellprepREP} and with $(\mathscr{U}^0,\mathscr{U},\varrho)\in H^n\times H^n\times H^n$ with $n\geq 4$.
\end{propal}
\begin{proof}
    See Appendix \ref{section:appendix2}.
\end{proof}

\section{Definition of the new modulated energy}
\label{section:defmod}
We give the basic definitions for the new modulated energy made with the stress-energy tensor (modulated stress-energy tensor).\\
\begin{propal}
\label{propal:tensordiffdefmod}Let $(\Phi^\varepsilon)_{0<\varepsilon<1}$ be solutions to \eqref{eq:KGPKGP} and $(\textbf{U},\rho)$ be a solution to \eqref{eq:EulerbaseREP} with well-prepared initial data from definition \ref{defi:wellprepREP}, let $T_{\alpha\beta}^{KG}$ and $T_{\alpha\beta}^{EP}$ be their respective stress-energy tensor field \ref{defi:energytensorKGP} and \ref{defi:energytensorREP}, then we have
    \begin{equation}
        T_{\alpha\beta}^{KG}-T_{\alpha\beta}^{EP}=h^\varepsilon_{\alpha\beta}+I^\varepsilon_{\alpha\beta},
    \end{equation}
with 
\begin{align*}
\label{hdefdefmod}
&h_{\alpha\beta}^\varepsilon=\frac{1}{2}((\varepsilon\textbf{D}_\alpha-i\textbf{U}_\alpha)\Phi^\varepsilon\overline{(\varepsilon\textbf{D}_\beta-i\textbf{U}_\beta)\Phi^\varepsilon}+\overline{(\varepsilon\textbf{D}_\alpha-i\textbf{U}_\alpha)\Phi^\varepsilon}(\varepsilon\textbf{D}_\beta-i\textbf{U}_\beta)\Phi^\varepsilon)\\
&-\frac{1}{2}g_{\alpha\beta}((\varepsilon\textbf{D}_\mu-i\textbf{U}_\mu)\Phi^\varepsilon\overline{(\varepsilon\textbf{D}^\mu-i\textbf{U}^\mu )\Phi^\varepsilon})-g_{\alpha\beta}(V(\rho^\varepsilon)-V(\rho)-V^{\prime}(\rho)(\rho^\varepsilon-\rho)),
\end{align*}
and 
\begin{equation}
\label{Idefdefmod}
I^\varepsilon_{\alpha\beta}=-\textbf{U}_\alpha \textbf{U}_\beta (\rho-\rho^\varepsilon)+\textbf{U}_\alpha (\textbf{J}^\varepsilon_\beta-\textbf{U}_\beta\rho^\varepsilon)+\textbf{U}_\beta (\textbf{J}^\varepsilon_\alpha-\textbf{U}_\alpha\rho^\varepsilon)-g_{\alpha\beta}(\textbf{J}^\varepsilon_\mu-\textbf{U}_\mu \rho^\varepsilon)\textbf{U}^\mu.\\
\end{equation}
\end{propal}
\begin{proof}
    By direct calculation and using the normalization \eqref{eq:propagwellprepREP}, we have 
\begin{align*}
	&\frac{\varepsilon^2}{2}(\textbf{D}_\alpha\Phi^\varepsilon\textbf{D}_\beta\overline{\Phi^\varepsilon}+\textbf{D}_\alpha\overline{\Phi^\varepsilon}\textbf{D}_\beta\Phi^\varepsilon)-\frac{1}{2}g_{\alpha\beta}(\varepsilon^2\textbf{D}_\mu\overline{\Phi^\varepsilon}\textbf{D}^\mu\Phi^\varepsilon+|\Phi^\varepsilon|^2+2V(|\Phi^\varepsilon|^2))-\rho \textbf{U}_\alpha \textbf{U}_\beta-g_{\alpha\beta}V(\rho)(\gamma-1)\\
&=\frac{\varepsilon^2}{2}(\textbf{D}_\alpha\Phi^\varepsilon\textbf{D}_\beta\overline{\Phi^\varepsilon}+\textbf{D}_\alpha\overline{\Phi^\varepsilon}\textbf{D}_\beta\Phi^\varepsilon)-2\textbf{U}_\alpha \textbf{U}_\beta |\Phi^\varepsilon|^2+\textbf{U}_\alpha \textbf{U}_\beta |\Phi^\varepsilon|^2-\textbf{U}_\alpha \textbf{U}_\beta (\rho-|\Phi^\varepsilon|^2)\\
&-\frac{1}{2}g_{\alpha\beta}(\varepsilon^2\textbf{D}_\mu\overline{\Phi^\varepsilon}\textbf{D}^\mu\Phi^\varepsilon+|\Phi^\varepsilon|^2+2V(|\Phi^\varepsilon|^2)-(\textbf{U}^\mu \textbf{U}_\mu+1+2V'(\rho))|\Phi^\varepsilon|^2+2V(\rho)(\gamma-1))\\
&=\frac{1}{2}((\varepsilon\textbf{D}_\alpha-i\textbf{U}_\alpha)\Phi^\varepsilon\overline{(\varepsilon\textbf{D}_\beta-i\textbf{U}_\beta)\Phi^\varepsilon}+\overline{(\varepsilon\textbf{D}_\alpha-i\textbf{U}_\alpha)\Phi^\varepsilon}(\varepsilon\textbf{D}_\beta-i\textbf{U}_\beta)\Phi^\varepsilon)+\textbf{J}^\varepsilon_\alpha \textbf{U}_\beta+\textbf{J}^\varepsilon_\beta \textbf{U}_\alpha\\
&-2\textbf{U}_\alpha \textbf{U}_\beta |\Phi^\varepsilon|^2-\textbf{U}_\alpha \textbf{U}_\beta (\rho-|\Phi^\varepsilon|^2)-\frac{1}{2}g_{\alpha\beta}((\varepsilon\textbf{D}_\mu-i\textbf{U}_\mu)\Phi^\varepsilon\overline{(\varepsilon\textbf{D}^\mu-i\textbf{U}^\mu )\Phi^\varepsilon}+2\textbf{J}^\varepsilon_\mu \textbf{U}^\mu -2\textbf{U}_\mu \textbf{U}^\mu |\Phi^\varepsilon|^2)\\
&-g_{\alpha\beta}(V(|\Phi^\varepsilon|^2)-V(\rho)-V^{\prime}(\rho)(|\Phi^\varepsilon|^2-\rho)).\\
\end{align*}
\end{proof}

\begin{defi}
\label{defi:etadefmod}
    Let $\textbf{X}$ be a time-like future-directed vector field and $h^\varepsilon_{\alpha\beta}$ be from \ref{propal:tensordiffdefmod}, we define $\eta(\textbf{X})^\varepsilon=h^\varepsilon_{\alpha0}\textbf{X}^\alpha$. It corresponds to looking at the modulated energy flux in the reference frame of $\textbf{X}$.\\
    \end{defi}
\begin{defi}
\label{defi:hdefintegralhdef}
    Let $(\Phi^\varepsilon)_{0<\varepsilon<1}$ be solutions to \eqref{eq:KGPKGP} and $(\textbf{U},\rho)$ be a solution to \eqref{eq:EulerbaseREP} with well-prepared initial data from definition \ref{defi:wellprepREP}, we define $H^\varepsilon$ as $\int_{\mathbb{R}^3}{\eta^\varepsilon(\textbf{U})dx}$, that is 
    \begin{align*}
    \label{eq:Hepsfulldefmod}
    H^\varepsilon(\Phi^\varepsilon,\partial_t\Phi^\varepsilon,\textbf{U},\rho)=\int_{\mathbb{R}^3}{\textbf{U}^\alpha\frac{1}{2}((\varepsilon\textbf{D}_\alpha-i\textbf{U}_\alpha)\Phi^\varepsilon\overline{(\varepsilon\textbf{D}_0-i\textbf{U}_0)\Phi^\varepsilon}+\overline{(\varepsilon\textbf{D}_\alpha-i\textbf{U}_\alpha)\Phi^\varepsilon}(\varepsilon\textbf{D}_0-i\textbf{U}_0)\Phi^\varepsilon)dx}\\
    -\int_{\mathbb{R}^3}{\frac{1}{2}\textbf{U}_0((\varepsilon\textbf{D}_\mu-i\textbf{U}_\mu)\Phi^\varepsilon\overline{(\varepsilon\textbf{D}^\mu-i\textbf{U}^\mu )\Phi^\varepsilon})+\textbf{U}_0(V(|\Phi^\varepsilon|^2)-V(\rho)-V^{\prime}(\rho)(|\Phi^\varepsilon|^2-\rho))dx}.
    \end{align*}
\end{defi}
\section{Proof of Theorem \ref{unTheorem:TH1mainth}}
\label{section:Proof}
The proof of \ref{unTheorem:TH1mainth} is divided in two parts. Firstly, we gather the results on the existence and the regularity of the solutions to the equations \eqref{eq:KGPKGP} and \eqref{eq:EulerbaseREP}. Secondly, we show the convergence of the momentum and the density (the semi-classical limit). This boils down to prove the \textbf{coercivity property} (point \ref{item:modpoint1Idea}) and the \textbf{propagation property} (point \ref{item:modpoint2Idea}) of the modulated energy.
\subsection{Existence and regularity of the solutions}
\label{subsection:reguexist}
\begin{propal}
\label{propal:reguexistProof}
Let $(\varphi^\varepsilon,\dot{\varphi}^\varepsilon)_{0<\varepsilon<1}$ and $(\mathscr{U}^0,\mathscr{U},\varrho)$ be initial data satisfying the assumptions \ref{item:item1mainth} \ref{item:item2mainth} and \ref{item:item3mainth}  of Theorem \ref{unTheorem:TH1mainth}. Then, point \ref{item:point1mainth} of the Theorem is true, the solutions $(\Phi^\varepsilon)_{0<\varepsilon<1}$ and $(\textbf{U},\rho)$ exist on $([0,t_\varepsilon])_{0<\varepsilon<1}$ and $[0,T]$ respectively and there exists $C_0(c_0)$ such that $\forall\;0<\varepsilon<1,\; 0\leq t\leq \min(t_\varepsilon,T)$
     \begin{equation}
     \label{eq:Coenergmainth}
         \int_{\mathbb{R}^3}{\frac{\varepsilon^2|\textbf{D}\Phi^\varepsilon(t)|^2}{2}+\frac{|\Phi^\varepsilon(t)|^2}{2}+V(|\Phi^\varepsilon(t)|^2)dx}+\sum_{j=0}^4||\textbf{U}||_{C^{4-j}([0,T],H^j)}^2+\sum_{j=0}^4||\rho||_{C^{4-j}([0,T],H^j)}\leq C_0.
     \end{equation}
Moreover, we can take $t_\varepsilon=T$ in the case $2\leq\gamma\leq3$.\\
\end{propal}
\begin{nota}
\label{nota:ColaconstantProof}
We note $C_0$ any constant that only depends on $C_0$.
\end{nota}
\begin{proof}
For the well posedness of KG \eqref{eq:KGPKGP} we use the classical energy method and Proposition \ref{propal:solutglobalKGP} in the case $2\leq\gamma\leq3$ ($t_\varepsilon$ is as large as we want). For the REP system \eqref{eq:EulerbaseREP} we use the Proposition \ref{propal:wellposedEulerEP}. The existence of $C_0$ is direct.  
\end{proof}
\subsection{\textbf{Coercivity property}, point \ref{item:modpoint1Idea}}
\label{subsection:Cyproof}
In this section, we consider that $(\Phi^\varepsilon)_{0<\varepsilon<1}$ and $(\textbf{U},\rho)$ are given by Proposition \ref{propal:reguexistProof} under the assumptions \ref{item:item1mainth}, \ref{item:item2mainth} and \ref{item:item3mainth} of Theorem \ref{unTheorem:TH1mainth}. Moreover, we consider $h^\varepsilon_{\alpha\beta}$ from Proposition \ref{propal:tensordiffdefmod}.\\
\begin{defi}
\label{defi:xiCyproof}
    We write\footnote{We write $\boldsymbol{\xi}$ instead of $\boldsymbol{\xi}^\varepsilon$ to lighten the notation.} $\boldsymbol{\xi}_\alpha=(\varepsilon\textbf{D}_\alpha-i\textbf{U}_\alpha)\Phi^\varepsilon$ and $\Theta^\varepsilon(|\Phi^\varepsilon|^2,\rho)=V(|\Phi^\varepsilon|^2)-V(\rho)-V^{\prime}(\rho)(|\Phi^\varepsilon|^2-\rho)$ so that 
\begin{equation}
\label{eq:defh2Cyproof}
h_{\alpha\beta}^\varepsilon=\frac{1}{2}(\boldsymbol{\xi}_\alpha\overline{\boldsymbol{\xi}_\beta}+\overline{\boldsymbol{\xi}_\alpha}\boldsymbol{\xi}_\beta)-\frac{1}{2}g_{\alpha\beta}(\boldsymbol{\xi}_\mu\overline{\boldsymbol{\xi}^\mu})-g_{\alpha\beta}\Theta^\varepsilon.
\end{equation}
\end{defi}

\begin{propal}
\label{propal:waystoxiCyproof}
We have 
\begin{equation}
\label{eq:h00simpleCyproof}
    h^\varepsilon_{00}=\frac{|\boldsymbol{\xi}|^2}{2}+\Theta^\varepsilon,
\end{equation}
and 
\begin{equation}
\label{eq:hgoodshapeCyproof}
h^\varepsilon_{00}=\varepsilon^2\frac{|\textbf{D}\sqrt{\rho^\varepsilon}|^2}{2}+\frac{|\textbf{J}^\varepsilon-\rho^\varepsilon \textbf{U}|^2}{2\rho^\varepsilon}+\Theta^\varepsilon.
\end{equation}    
\end{propal}
\begin{proof}
We have \eqref{eq:h00simpleCyproof} by direct calculation. The equation \eqref{eq:hgoodshapeCyproof} is given in \cite{zbMATH06101438}, we give a proof of this identity for the sake of completeness. We have
\begin{align*}
    &\frac{|(\varepsilon\textbf{D}-i\textbf{U})\Phi^\varepsilon|^2}{2}=\frac{|\varepsilon\textbf{D}\Phi^\varepsilon|^2}{2}+\frac{|\textbf{U}|^2|\Phi^\varepsilon|^2}{2}-\textbf{J}^\varepsilon\cdot\textbf{U}=\frac{\textbf{J}^\varepsilon\cdot\textbf{J}^\varepsilon}{2\rho^\varepsilon}+\frac{\varepsilon^2\textbf{D}\sqrt{\rho^\varepsilon}\cdot\textbf{D}\sqrt{\rho^\varepsilon}}{2}+\frac{|\textbf{U}|^2\rho^\varepsilon}{2}-\textbf{J}^\varepsilon\cdot\textbf{U}\\    &=\varepsilon^2\frac{|\textbf{D}\sqrt{\rho^\varepsilon}|^2}{2}+\frac{|\textbf{J}^\varepsilon-\rho^\varepsilon \textbf{U}|^2}{2\rho^\varepsilon},
\end{align*}
where we use the equation \eqref{eq:splitJJ2KGP}. \\
\end{proof}
\begin{propal}
\label{propal:c1Cyproof}
     There exists $c_1>0$ such that 
\begin{equation}
\label{eq:c1Cyproof}
		\forall\alpha,\beta\in[\![ 0,3]\!],   \;h^\varepsilon_{00}> c_1|h^\varepsilon_{\alpha\beta}|.
\end{equation}
\end{propal}
\begin{proof}
By direct calculation.\\
\end{proof}
\begin{propal}
\label{prop:propaltheta1Cyproof}
    For $x\geq 0$ $y\geq 0$ the function $\Theta(x,y)$ is positive, with
    \begin{equation}
    \label{eq:inegtheta11Cyproof}
        \Theta(x,y)\gtrsim (x-y)^2(x^{\gamma-2}+y^{\gamma-2}),
    \end{equation}
    and 
    \begin{equation}
    \label{eq:inegtheta12Cyproof}
    \Theta(x,y)\gtrsim |x-y|^\gamma.
    \end{equation}
\end{propal}
\begin{proof}
    These inequalities are used in \cite{zbMATH06101438}, we give here their proofs for the sake of completeness. For the first inequality \eqref{eq:inegtheta11Cyproof}, we search for $c>0$ such that $\Theta(x,y)\geq c(x-y)^2(x^{\gamma-2}+y^{\gamma-2})$.  The inequality is direct for $y=0$, so we divide everything by $y^\gamma$ to get the following inequality for $z=\frac{x}{y}$ 
    \begin{align*}
       \frac{z^\gamma}{\gamma}-\frac{1}{\gamma}-(z-1)\geq c(z-1)^2(z^{\gamma-2}+1).
    \end{align*}
   We split it into two parts. Firstly, we have
    \begin{align}
    \label{eq:inegz1Cyproof}
       \frac{z^\gamma}{\gamma}-\frac{1}{\gamma}-(z-1)\geq \frac{1}{2}(z-1)^2.
    \end{align}
   Indeed, the inequality is true at $z=1$ and by taking a derivative with respect to $z$ we find
    \begin{align*}
      &\forall z\in[0,1),\;\; z^{\gamma-1}-1\leq z-1, \\
      &\forall z>1,\;\;  z^{\gamma-1}-1\geq z-1.
    \end{align*}
   More generally, this implies that for any $0<c<\frac{1}{\gamma}\leq\frac{1}{2}$ we have 
     \begin{align}
     \label{eq:inegz2Cyproof}
       &\frac{z^\gamma}{\gamma}-\frac{1}{\gamma}-(z-1)\geq c(z-1)^2,
    \end{align}
    and also
     \begin{align*}
       &\forall z\in[0,1),\;\;\frac{z^\gamma}{\gamma}-\frac{1}{\gamma}-(z-1)\geq c(z-1)^2z^{\gamma-2}.
       \end{align*}
       It remains to prove that $\exists c>0$ such that
           \begin{align*}
       &\forall z>1,\;\;\frac{z^\gamma}{\gamma}-\frac{1}{\gamma}-(z-1)\geq c(z-1)^2z^{\gamma-2}.
       \end{align*}
    We compute 
    \begin{align*}
    &\partial_z(\frac{z^\gamma}{\gamma}-\frac{1}{\gamma}-(z-1)-c(z-1)^2z^{\gamma-2})\\
        &=z^{\gamma-1}-1-2c(z-1)z^{\gamma-2}-c(\gamma-2)z^{\gamma-3}(z-1)^2\\
        &=z^{\gamma-2}(z(1-c\gamma)+c\gamma)+z^{\gamma-3}(z-1)c(\gamma-2)-1.
    \end{align*}
  We know that $z^{\gamma-3}(z-1)c(\gamma-2)\geq0$ and $z^{\gamma-2}(z(1-c\gamma)+c\gamma)>1$ for $z>1$, $\gamma\geq2$ and $c<\frac{1}{\gamma}$.
  We deduce that the derivative is strictly positive and so for $c<\frac{1}{\gamma}$
  \begin{align}
  \label{eq:inegz3Cyproof}
       &\frac{z^\gamma}{\gamma}-\frac{1}{\gamma}-(z-1)\geq c(z-1)^2z^{\gamma-2}.
  \end{align}
  The two inequalities \eqref{eq:inegz2Cyproof} and \eqref{eq:inegz3Cyproof} give us the inequality  \eqref{eq:inegtheta11Cyproof} for $c<\frac{1}{2\gamma}$.
  For the second inequality, we just need to show that for some $c>0$
   \begin{align*}
     (x-y)^2(x^{\gamma-2}+y^{\gamma-2})\geq c(x-y)^\gamma.
    \end{align*}
  It is trivial for $y=0$ so we do the same reasoning with $z=\frac{x}{y}$. With $c=1$ we directly have 
  \begin{align*}
  (z^{\gamma-2}+1)\geq c(z-1)^{\gamma-2}.
    \end{align*}
\end{proof}
\begin{propal}
\label{propal:inegtheta2Cyproof}
Let $f,g\in L^\gamma$ be two positive non-zero\footnote{There is actually an equality if at least one of the functions is zero.} functions, then we have 
\begin{align}
\label{eq:inegtheta21Cyproof}
    ||f^\gamma-g^\gamma||_{L^1}\leq C_\gamma(||f||_{L^{\gamma}},||g||_{L^{\gamma}})||f-g||^\delta_{L^\gamma}
\end{align}
and  
\begin{align}
\label{eq:inegtheta22Cyproof}
    ||f^\gamma-g^\gamma||_{L^1}\leq C_\gamma(||f||_{L^{\gamma}},||g||_{L^{\gamma}})||\Theta(f,g)||^{\frac{\delta}{\gamma}}_{L^1},
\end{align}
where $\delta>0$ and $C_\gamma:\mathbb{R}^\star_+\times\mathbb{R}^\star_+\to\mathbb{R}^\star_+$ only depend on $\gamma$.
\end{propal}
\begin{proof}
First, we consider $\gamma=n$ an integer. Let $x,y\geq0$, then we have
\begin{itemize}
\item $|x^n-y^n|=|(x^{n/2}-y^{n/2})^2+2y^{n/2}(x^{n/2}-y^{n/2})|$ for $n$ even,
\item $|x^n-y^n|=|(x^{(n-1)/2}-y^{(n-1)/2})^2(x-y)-yx(y^{n-2}-x^{n-2})-2x^{(n-1)/2}y^{(n-1)/2}(y-x)|$ for $n$ odd and strictly greater than $1$,
\end{itemize}
so that by induction we directly get 
\begin{align}
\label{eq:inegxyCyproof}
|x^n-y^n|\leq \sum_{p,q,r\in\mathbb{N},\;r\geq1,\;r+q+p=n}c_{r,p,q}|x-y|^rx^py^q,
\end{align}
where the $c_{r,p,q}$ are constant.\\
Now, for a general $\gamma=n+\epsilon$ with $n=\floor{\gamma}{}$ and $0\leq\epsilon<1$ we have 
\begin{align*}
    |x^\gamma-y^\gamma|=|(x^n-y^n)x^\epsilon+(x^\epsilon-y^\epsilon)y^n|\leq |x^n-y^n|x^\epsilon+|x-y|^\epsilon y^n,
\end{align*}
and so with \eqref{eq:inegxyCyproof} we have 
\begin{align*}
    |x^\gamma-y^\gamma|\leq \sum_{a,b,c\in\mathbb{R}_+,\;a>0,\;a+b+c=\gamma}c_{a,b,c}|x-y|^ax^by^c,
\end{align*}
where the number of $c_{a,b,c}$ which are non-zero is clearly finite. Let $f,g\in L^\gamma$ be two positive functions then
\begin{align*}
    ||f^\gamma-g^\gamma||_{L^1}\leq \sum_{a,b,c\in\mathbb{R}_+,\;a>0,\;a+b+c=\gamma}c_{a,b,c}||(|f-g|^af^bg^c)||_{L^1}\leq C_\gamma(||f||_{L^{\gamma}},||g||_{L^{\gamma}})||f-g||^\delta_{L^\gamma},
\end{align*}
for some $\delta>0$ and some $C_\gamma:\mathbb{R}^\star_+\times\mathbb{R}^\star_+\to\mathbb{R}^\star_+$, using the Hölder inequality. The second inequality \eqref{eq:inegtheta22Cyproof} comes from the previous result and \eqref{eq:inegtheta12Cyproof}.\\
\end{proof}
\begin{propal}
\label{propal:tryoneh00Cyproof}
    Let $H^\varepsilon_0=\int_{\mathbb{R}^3}{h_{00}^\varepsilon dx}$, as in \eqref{eq:moduH0}, then it satisfies the \textbf{coercivity property}, point \ref{item:modpoint1Idea}.
\end{propal}
\begin{proof}
    By algebraic computation, from
  \eqref{eq:inegtheta12Cyproof} we have for $\gamma\geq2$ 
\begin{equation}
\label{eq:convtheta1osefCyproof}
    ||\rho^\varepsilon-\rho||_{L^\gamma}\leq (H^\varepsilon_0)^{1/\gamma},
\end{equation}
and from \eqref{eq:inegtheta22Cyproof}
\begin{equation}
\label{eq:convtheta2osefCyproof}
    ||V(\rho^\varepsilon)-V(\rho)||_{L^1}\leq\frac{1}{\gamma}||(\rho^\varepsilon)^\gamma-(\rho)^\gamma||_{L^1}\leq C_0(H^\varepsilon_0)^{\delta/\gamma},
\end{equation}
for some $\delta>0$ and where we control the $L^\gamma$ norms in $O(1)$ with the energy \eqref{eq:Coenergmainth} for $\rho^\varepsilon$ and with classical Sobolev embedding for $\rho$. Moreover,
\begin{align*}
     ||\textbf{J}^\varepsilon-\textbf{U}\rho||_{L^{\frac{2\gamma}{\gamma+1}}}&\leq  ||\textbf{J}^\varepsilon-\textbf{U}\rho^\varepsilon||_{L^{\frac{2\gamma}{\gamma+1}}}+||\textbf{U}\rho^\varepsilon-\textbf{U}\rho||_{L^{\frac{2\gamma}{\gamma+1}}},\\
     &\leq||\sqrt{\rho^\varepsilon}||_{L^{2\gamma}}||\frac{\textbf{J}^\varepsilon-\textbf{U}\rho^\varepsilon}{\sqrt{\rho^\varepsilon}}||_{L^{2}}+||\rho^\varepsilon-\rho||_{L^{\gamma}}||\textbf{U}||_{L^{\frac{2\gamma}{\gamma-1}}}\leq C_0((H^\varepsilon_0)^{1/2}+(H^\varepsilon_0)^{\frac{1}{\gamma}}),
\end{align*}
and
\begin{align*}
     ||\textbf{J}^\varepsilon-\textbf{U}\rho||_{L^1}&\leq  ||\textbf{J}^\varepsilon-\textbf{U}\rho^\varepsilon||_{L^{1}}+||\textbf{U}\rho^\varepsilon-\textbf{U}\rho||_{L^{1}}\\    &\leq||\Im(\Phi^\varepsilon\overline{(\varepsilon\textbf{D}-i\textbf{U})\Phi^\varepsilon})||_{L^{1}}+||\textbf{U}||_{L^{\frac{\gamma}{\gamma-1}}}||\rho^\varepsilon-\rho||_{L^{\gamma}} \\
     &\leq ||\Phi^\varepsilon||_{L^2}||(\varepsilon\textbf{D}-i\textbf{U})\Phi^\varepsilon||_{L^{2}}+||\textbf{U}||_{L^{\frac{\gamma}{\gamma-1}}}||\rho^\varepsilon-\rho||_{L^{\gamma}}\leq C_0((H^\varepsilon_0)^{1/2}+(H^\varepsilon_0)^{\frac{1}{\gamma}}).
\end{align*}
Finally, for some $\delta>0$, we clearly recover 
\begin{equation}
\label{eq:convtryoneoasefCyproof}
H^\varepsilon_0=O(\varepsilon^2)\implies\lim_{\varepsilon\to 0}(||\textbf{J}^\varepsilon-\textbf{U}\rho||_{L^{\frac{2\gamma}{\gamma+1}}}+||\textbf{J}^\varepsilon-\textbf{U}\rho||_{L^1}+||\rho^\varepsilon-\rho||_{L^\gamma}+||V(\rho^\varepsilon)-V(\rho)||_{L^1})=\lim_{\varepsilon\to 0}O(\varepsilon^{\delta})=0.
 \end{equation}
\end{proof}
\begin{rem}
\label{rem:convdensityneedCyproof}
We can observe that the convergence of the density $\rho^\varepsilon$ totally relies on the presence of the potential $V'$. Moreover, the convergence of the momentum is also directly linked to the convergence of the density and thus to $V'$. In particular, if $V'=0$ we can see that the arguments of the proof of the previous Proposition \ref{propal:tryoneh00Cyproof} do not apply anymore.\\
\end{rem}
\begin{propal}
\label{propal:c2c3Cyproof}
    Let $\textbf{X}$ be a time-like future-directed vector field such that $\textbf{X}$ is uniformly bounded with $\textbf{X}^0\geq\nu$ and $-\textbf{X}^\alpha \textbf{X}_\alpha\geq\nu$ for some $\nu>0$, then there exists $c_2(\textbf{X})$ and $c_3(\textbf{X})$ such that for $\eta^\varepsilon(\textbf{X})=h^\varepsilon_{0\beta}\textbf{X}^\beta$ we have  
\begin{equation}
\label{eq:c2c3Cyproof}
		c_2(\textbf{X})h_{00}\leq\eta(\textbf{X})\leq c_3(\textbf{X})h_{00}.
\end{equation}
\end{propal}
\begin{proof}
  We use the notation \ref{nota:gradNOTA2}, we have 
\begin{align*}
&h_{0\beta}\textbf{X}^\beta=h_{00}X^0+h_{0i}X^i=\frac{|\boldsymbol{\xi}|^2}{2}X^0+\frac{\xi_0\overline{\xi_i}+\overline{\xi_0}\xi_i}{2}X^i+\Theta X^0\\
&\geq \frac{|\boldsymbol{\xi}|^2}{2}X^0+\Theta X^0-|\xi_0||\xi||X|\geq \frac{|\boldsymbol{\xi}|^2}{2}(X^0-|X|)+\Theta X^0\\
&\geq c_2(\nu)h_{00}^\varepsilon,
\end{align*}
where we used the Young inequality. We have the existence of $c_2(\textbf{X})$. \\
The existence of $c_3(\textbf{X})$ is a consequence of Proposition \ref{propal:c1Cyproof} and the fact that $\textbf{X}$ is uniformly bounded.\\
\end{proof}
\begin{propal}
\label{propal:HgoodfinalCyproof}
 Let $H^\varepsilon=\int_{\mathbb{R}^3}{\eta^\varepsilon(\textbf{U})dx}$ (defined in \ref{defi:etadefmod}) then it satisfies the coercivity property point \ref{item:modpoint1Idea}. The new modulated energy $H^\varepsilon$ is equivalent to $H^\varepsilon_0$ in terms of coercivity. We have\footnote{See definition \ref{defi:xiCyproof} and Remark \ref{propal:waystoxiCyproof}.}
 \begin{align}
     \label{eq:equivHH0Cyproof}
     H^\varepsilon\sim H^\varepsilon_0=\int_{\mathbb{R}^3}{h^\varepsilon_{00}dx}=\int_{\mathbb{R}^3}{\frac{|\boldsymbol{\xi}|^2}{2}+\Theta^\varepsilon dx}= \int_{\mathbb{R}^3}{\varepsilon^2\frac{|\textbf{D}\sqrt{\rho^\varepsilon}|^2}{2}+\frac{|\textbf{J}^\varepsilon-\rho^\varepsilon  \textbf{U}|^2}{2\rho^\varepsilon}+\Theta^\varepsilon dx}.
 \end{align}
\end{propal}
\begin{proof}
    For the vector field $\textbf{U}$, we have $\nu=1$ and $||\textbf{U}||_{C^0}\leq C_0$ so that we can apply Proposition \ref{propal:c2c3Cyproof} to get 
    \begin{equation}
		c_2(\textbf{U})H^\varepsilon_{0} \leq H^\varepsilon\leq c_3(\textbf{U})H^\varepsilon_{0},
\end{equation}
then we use the coercivity property of $H^\varepsilon_0$ from \ref{propal:tryoneh00Cyproof}.\\
\end{proof}
\begin{rem}
\label{rem:shiftedframeCyproof}
The quantity $\eta(\textbf{U})$ corresponds to the modulated energy flux in the reference frame of $\textbf{U}$.
\end{rem}
\subsection{\textbf{Propagation property}, point \ref{item:modpoint2Idea}}
\label{subsection:Pgproof}
In this section, we consider that $(\Phi^\varepsilon)_{0<\varepsilon<1}$ and $(\textbf{U},\rho)$ are given by Proposition \ref{propal:reguexistProof} under the assumptions \ref{item:item1mainth}, \ref{item:item2mainth}, \ref{item:item3mainth} and \ref{item:item4mainth} of Theorem \ref{unTheorem:TH1mainth}.\\
\begin{propal}
\label{propal:finalPgproof}
 Let $H^\varepsilon=\int_{\mathbb{R}^3}{\eta^\varepsilon(\textbf{U})dx}$ (defined in \ref{defi:hdefintegralhdef}) then it satisfies point \ref{item:modpoint2Idea}. The new modulated energy $H^\varepsilon$ propagates its smallness. 
\end{propal}
\begin{proof}
First, we calculate that 
\begin{align}
\label{eq:etaUPgproof}
&\eta^\varepsilon(\textbf{U})=h^\varepsilon_{0\alpha}\textbf{U}^\alpha=(T_{0\alpha}^{KG}-T_{0\alpha}^{EP})\textbf{U}^\alpha -I^\varepsilon_{0\alpha}\textbf{U}^\alpha \\
&=(T_{0\alpha}^{KG}-T_{0\alpha}^{EP})\textbf{U}^\alpha-\textbf{U}_0 \textbf{U}_\alpha (\rho-\rho^\varepsilon)\textbf{U}^\alpha +\textbf{U}_0 (\textbf{J}^\varepsilon_\alpha-\textbf{U}_\alpha\rho^\varepsilon)\textbf{U}^\alpha \nonumber\\
&+\textbf{U}_\alpha (\textbf{J}^\varepsilon_0-\textbf{U}_0\rho^\varepsilon)\textbf{U}^\alpha -\textbf{U}_0(\textbf{J}^\varepsilon_\mu-\textbf{U}_\mu \rho^\varepsilon)\textbf{U}^\mu\nonumber\\
&=(T_{0\alpha}^{KG}-T_{0\alpha}^{EP})\textbf{U}^\alpha-\textbf{U}_\alpha\textbf{U}^\alpha (\textbf{U}_0 \rho-\textbf{J}^\varepsilon_0) \nonumber\\
&=(T_{0\alpha}^{KG}-T_{0\alpha}^{EP})\textbf{U}^\alpha +(J^\varepsilon_0-\textbf{U}_0\rho)+2V'(\rho)(J^\varepsilon_0-\textbf{U}_0\rho).\nonumber
\end{align}
using the definition \ref{defi:etadefmod} and the Proposition \ref{propal:propagwellprepREP}. Then, we compute the derivative 
\begin{align*}
		&\frac{d}{dt}H^\varepsilon=\frac{d}{dt}\int_{\mathbb{R}^3}{\eta^\varepsilon(\textbf{U})dx}\\
  &=-\int_{\mathbb{R}^3}{(T_{\alpha\beta}^{KG}-T_{\alpha\beta}^{EP})\textbf{D}^\alpha \textbf{U}^\beta dx}+2\int_{\mathbb{R}^3}{\partial_t\rho V^{\prime\prime}(\rho)(\textbf{J}^\varepsilon_0-\textbf{U}_0\rho) dx}+2\int_{\mathbb{R}^3}{ V^{\prime}(\partial_t\textbf{J}^\varepsilon_0-\partial_t(\textbf{U}_0\rho)) dx}\\
  &=-\int_{\mathbb{R}^3}{(h^\varepsilon_{\alpha\beta}+I^\varepsilon_{\alpha\beta})\textbf{D}^\alpha \textbf{U}^\beta dx}+2\int_{\mathbb{R}^3}{\partial_tV^{\prime}(\rho)(\textbf{J}^\varepsilon_0-\textbf{U}_0\rho) dx}
+2\int_{\mathbb{R}^3}{ V^{\prime}(\partial_t\textbf{J}^\varepsilon_0-\partial_t(\textbf{U}_0\rho)) dx},
\end{align*}
by the fact that $T_{\alpha0}^{KG}$ and $T_{\alpha0}^{EP}$ have $0$ divergence and the fact that the charges are conserved, see \eqref{eq:divstressenergyKGP}, \eqref{eq:divstressenergyREP}, \eqref{eq:conservedchargeKGP} and \eqref{eq:conservchargeREP}.\\
Moreover, because $\textbf{U}$ and $\rho$ are solutions\footnote{In particular, we use the derivatives of the normalization \eqref{eq:propagwellprepREP}.} to \eqref{eq:EulerbaseREP} we have 
\begin{align*}
    &-\int_{\mathbb{R}^3}{I^\varepsilon_{\alpha\beta}\textbf{D}^\alpha \textbf{U}^\beta dx}\\
    &=-\int_{\mathbb{R}^3}{-\textbf{U}_\alpha \textbf{U}_\beta\textbf{D}^\alpha \textbf{U}^\beta(\rho-\rho^\varepsilon)+\textbf{U}_\alpha (\textbf{J}^\varepsilon_\beta-\textbf{U}_\beta\rho^\varepsilon)(\textbf{D}^\alpha \textbf{U}^\beta+\textbf{D}^\beta \textbf{U}^\alpha)-\textbf{D}_\alpha \textbf{U}^\alpha(\textbf{J}^\varepsilon_\mu-\textbf{U}_\mu \rho^\varepsilon)\textbf{U}^\mu dx}\\
    &=-\int_{\mathbb{R}^3}{\textbf{U}^\beta \textbf{D}_\beta V'(\rho)(\rho-\rho^\varepsilon)-2(\textbf{J}^\varepsilon_\beta-\textbf{U}_\beta\rho^\varepsilon)\textbf{D}^\beta V'(\rho)-\textbf{D}_\alpha \textbf{U}^\alpha(\textbf{J}^\varepsilon_\mu-\textbf{U}_\mu \rho^\varepsilon)\textbf{U}^\mu dx}\\
    &=-\int_{\mathbb{R}^3}{-2\textbf{J}^\varepsilon_\beta \textbf{D}^\beta V'(\rho)+\textbf{U}^\beta \textbf{D}_\beta V'(\rho)(\rho+\rho^\varepsilon)-\textbf{D}_\alpha \textbf{U}^\alpha(\textbf{J}^\varepsilon_\mu-\textbf{U}_\mu \rho^\varepsilon)\textbf{U}^\mu dx},
\end{align*}
and because the charges \eqref{eq:conservedchargeKGP} and \eqref{eq:conservchargeREP} are conserved we have
\begin{align*}
    2\int_{\mathbb{R}^3}{\partial_tV^{\prime}(\rho)(\textbf{J}^\varepsilon_0-\textbf{U}_0\rho) dx}
+2\int_{\mathbb{R}^3}{ V^{\prime}(\partial_t\textbf{J}^\varepsilon_0-\partial_t(\textbf{U}_0\rho)) dx}&=2\int_{\mathbb{R}^3}{\partial_tV^{\prime}(\rho)(\textbf{J}^\varepsilon_0-\textbf{U}_0\rho) dx}\\
&+2\int_{\mathbb{R}^3}{ V^{\prime}(\textbf{D}^i \textbf{J}^\varepsilon_i-\textbf{D}^i(\textbf{U}_i\rho)) dx}\\
&=-2\int_{\mathbb{R}^3}{\textbf{D}^\alpha V^{\prime}(\rho)(\textbf{J}^\varepsilon_\alpha-\textbf{U}_\alpha\rho) dx}.
\end{align*}
Finally, we have
\begin{align*}
		&\frac{d}{dt}H^\varepsilon=-\int_{\mathbb{R}^3}{h^\varepsilon_{\alpha\beta}\textbf{D}^\alpha \textbf{U}^\beta dx}-\int_{\mathbb{R}^3}{\textbf{U}^\beta \textbf{D}_\beta V'(\rho)(\rho^\varepsilon-\rho)dx}+\int_{\mathbb{R}^3}{\textbf{D}_\alpha \textbf{U}^\alpha(\textbf{J}^\varepsilon_\mu-\textbf{U}_\mu \rho^\varepsilon)\textbf{U}^\mu dx}\\
  &=-\underbrace{\int_{\mathbb{R}^3}{h^\varepsilon_{\alpha\beta}\textbf{D}^\alpha \textbf{U}^\beta dx}}_\text{$\mathscr{H}_1$}+\underbrace{\int_{\mathbb{R}^3}{\rho \textbf{D}_\beta \textbf{U}^\beta V''(\rho)(\rho^\varepsilon-\rho)dx}}_\text{$\mathscr{H}_2$}+\underbrace{\int_{\mathbb{R}^3}{\textbf{D}_\alpha \textbf{U}^\alpha(\textbf{J}^\varepsilon_\mu-\textbf{U}_\mu \rho^\varepsilon)\textbf{U}^\mu dx}}_\text{$\mathscr{H}_3$}.\\
\end{align*}
We  know that we have
\begin{equation}
\label{eq:control1parHPgproof}
    \mathscr{H}_1\leq C_0 H^\varepsilon,
\end{equation}
using the previous calculations and the Propositions \ref{propal:c2c3Cyproof} and \ref{propal:HgoodfinalCyproof}.\\
At first sight, it looks like the two last terms, $\mathscr{H}_2$ and $\mathscr{H}_3$, are not controllable by $H^\varepsilon$ but only by $(H^\varepsilon)^{1/\gamma}$ and $(H^\varepsilon)^{1/2}$, they are only linear in the quantities that we are looking at.
In fact, because of the specific structure of these terms we can improve the apparent lack of smallness.\\
Firstly,
\begin{align*}
\mathscr{H}_3=\underbrace{\int_{\mathbb{R}^3}{\textbf{D}_\alpha \textbf{U}^\alpha \frac{(\textbf{J}^\varepsilon_\mu-\textbf{U}_\mu \rho^\varepsilon)(\textbf{U}^\mu \rho^\varepsilon-(\textbf{J}^\varepsilon)^\mu)}{2\rho^\varepsilon}dx}}_\text{$\mathscr{H}_{3.1}$}+\underbrace{\int_{\mathbb{R}^3}{\textbf{D}_\alpha \textbf{U}^\alpha (\frac{\textbf{J}^\varepsilon_\mu(\textbf{J}^\varepsilon)^\mu}{2\rho^\varepsilon}-\frac{\textbf{U}_\mu \textbf{U}^\mu \rho^\varepsilon}{2})dx}}_\text{$\mathscr{H}_{3.2}$},
\end{align*}
and with Proposition \ref{propal:HgoodfinalCyproof} we directly get 
\begin{align}
\label{eq:control2parHPgproof}
\mathscr{H}_{3.1}\leq C_0H^\varepsilon.
\end{align}
For $\mathscr{H}_{3.2}$, we use the normalization \eqref{eq:propagwellprepREP} and the equation \eqref{eq:splitJJ3KGP} to get 
\begin{align*}
    &\mathscr{H}_{3.2}=\int_{\mathbb{R}^3}{\textbf{D}_\alpha \textbf{U}^\alpha(\frac{\varepsilon^2\sqrt{\rho^\varepsilon}\Box\sqrt{\rho^\varepsilon}-\rho^\varepsilon-2V'(\rho^\varepsilon)\rho^\varepsilon}{2}+\frac{\rho^\varepsilon+2V'(\rho)\rho^\varepsilon}{2})dx}\\
    &=\underbrace{\int_{\mathbb{R}^3}{\textbf{D}_\alpha \textbf{U}^\alpha\frac{\varepsilon^2\sqrt{\rho^\varepsilon}\Box\sqrt{\rho^\varepsilon}}{2}dx}}_\text{$\mathscr{H}_{3.2.1}$}\underbrace{-\int_{\mathbb{R}^3}{\textbf{D}_\alpha \textbf{U}^\alpha (V'(\rho^\varepsilon)\rho^\varepsilon-V'(\rho)\rho^\varepsilon)dx}}_\text{$\mathscr{H}_{3.2.2}$}.
\end{align*}
Then, we get 
\begin{align*}
    &\mathscr{H}_{3.2.1}=-\frac{d}{dt}\int_{\mathbb{R}^3}{\textbf{D}_\alpha \textbf{U}^\alpha\frac{\varepsilon^2\sqrt{\rho^\varepsilon}\partial_t \sqrt{\rho^\varepsilon}}{2}dx}+\int_{\mathbb{R}^3}{\textbf{D}_\alpha \textbf{U}^\alpha\frac{\varepsilon^2\partial_t \sqrt{\rho^\varepsilon}\partial_t \sqrt{\rho^\varepsilon}}{2}dx}+\int_{\mathbb{R}^3}{\partial_t\textbf{D}_\alpha \textbf{U}^\alpha\frac{\varepsilon^2\sqrt{\rho^\varepsilon}\partial_t \sqrt{\rho^\varepsilon}}{2}dx}\\
    &-\int_{\mathbb{R}^3}{\textbf{D}_\alpha \textbf{U}^\alpha\frac{\varepsilon^2\nabla \sqrt{\rho^\varepsilon}\nabla \sqrt{\rho^\varepsilon}}{2}dx}-\int_{\mathbb{R}^3}{\nabla\textbf{D}_\alpha \textbf{U}^\alpha\nabla \sqrt{\rho^\varepsilon}\frac{\varepsilon^2 \sqrt{\rho^\varepsilon}}{2}dx}\\
    &\leq -\frac{d}{dt}\int_{\mathbb{R}^3}{\textbf{D}_\alpha \textbf{U}^\alpha\frac{\varepsilon^2\sqrt{\rho^\varepsilon}\partial_t \sqrt{\rho^\varepsilon}}{2}dx}+C_0||\varepsilon \textbf{D} \sqrt{\rho^\varepsilon}||^2_{L^2}+\varepsilon C_0||\varepsilon \textbf{D} \sqrt{\rho^\varepsilon}||_{L^2}||\sqrt{\rho^\varepsilon}||_{L^2}\\
    &\leq \frac{d}{dt}\underbrace{\int_{\mathbb{R}^3}{-\textbf{D}_\alpha \textbf{U}^\alpha\frac{\varepsilon^2\sqrt{\rho^\varepsilon}\partial_t \sqrt{\rho^\varepsilon}}{2}dx}}_\text{$G^\varepsilon$}+C_0H^\varepsilon+\varepsilon C_0(H^\varepsilon)^{1/2},
\end{align*}
where the mass term in the energy \eqref{eq:EenergyKGP} controls $||\sqrt{\rho^\varepsilon}||_{L^2}$ by $C_0$ with \eqref{eq:Coenergmainth}, and by Proposition \ref{propal:HgoodfinalCyproof}
\begin{align*}
    &\mathscr{H}_{2}+\mathscr{H}_{3.2.2}=\int_{\mathbb{R}^3}{\textbf{D}_\beta \textbf{U}^\beta(\rho V''(\rho)(\rho^\varepsilon-\rho)-V'(\rho^\varepsilon)\rho^\varepsilon+V'(\rho)\rho^\varepsilon)dx}\\
    &=\int_{\mathbb{R}^3}{\textbf{D}_\beta \textbf{U}^\beta((\gamma-1) V'(\rho)(\rho^\varepsilon-\rho)-\gamma V(\rho^\varepsilon)+V'(\rho)\rho^\varepsilon)dx}=\int_{\mathbb{R}^3}{\gamma\textbf{D}_\beta \textbf{U}^\beta(V'(\rho)(\rho^\varepsilon-\rho)-V(\rho^\varepsilon)+ V(\rho))dx}\\
    &=-\int_{\mathbb{R}^3}{\gamma\textbf{D}_\beta \textbf{U}^\beta\Theta^\varepsilon(\rho^\varepsilon,\rho)dx}\leq C_0H^\varepsilon,
\end{align*}
where, in the last calculation, we use the fact that $V$ is a power law\footnote{In \cite{zbMATH05243173}, this part requires more arguments because the author deal with an approximate solution, i.e., a solution to the equation with a modified nonlinearity that is not a power law, and then pass to the limit. This is done to recover the global solutions, here this result is direct.}. 
Gathering everything, we have 
\begin{equation}
\label{eq:ineqHGgoodPgproof}
		\frac{d}{dt}H^\varepsilon\leq C_0\left(H^\varepsilon+(H^\varepsilon)^{1/2}\varepsilon\right)+\frac{d}{dt}G^\varepsilon.
\end{equation}
We integrate between $0$ and $t\in[0,\min(t_\varepsilon,T)]$ and use the Young inequality to get 
\begin{align*}
		H^\varepsilon(t)\leq H^\varepsilon(0)+C_0\int^{t}_0{\left(H^\varepsilon(s)+\varepsilon^2 \right)ds}+G^\varepsilon(t)-G^\varepsilon(0).
\end{align*}
Now, using again the Young inequality properly and calculations that are similar to the previous ones, we find that for all $t\in[0,\min(t_\varepsilon,T)]$
\begin{equation*}
		G^\varepsilon(t)\leq C_0(\delta^{-1}\varepsilon^2+\delta H^\varepsilon(t)),
\end{equation*}
for $\delta$ as small as we want. It gives us
\begin{equation}
\label{eq:almostHgoodPgproof}
				H^\varepsilon(t)\leq H^\varepsilon(0)+C_0(\int^{t}_0{\left(H^\varepsilon(s)+\varepsilon^2 \right) ds}+\delta H^\varepsilon(t)+\varepsilon^2+\delta H^\varepsilon(0)),
\end{equation}
and so for $\delta$ small in comparison to $C_0$ we have
\begin{equation}
\label{eq:finalHPgproof}
		H^\varepsilon(t)\leq C_0( H^\varepsilon(0)+\int^{t}_0{\left(H^\varepsilon(s)+\varepsilon^2 \right)ds}+\varepsilon^2).
\end{equation}
By Gronwall lemma and the assumptions of Theorem \ref{unTheorem:TH1mainth} we have
\begin{equation}
\label{eq:finalHsmallPgproof}
		H^\varepsilon(t)\leq C_0\varepsilon^2,
\end{equation}
for all $t\in[0,\min(t_\varepsilon,T)]$ and thus the \textbf{propagation property}, point \ref{item:modpoint2Idea}.
    
\end{proof}
\subsection{Putting everything together}
From section \ref{subsection:reguexist} we have point \ref{item:point1mainth} of the main Theorem. Then, with the results of section \ref{subsection:Cyproof} and \ref{subsection:Pgproof} we get point \ref{item:point2mainth}. The smallness of the modulated energy propagates and implies the semiclassical (monokinetic) limit of KGP to the REP system. Moreover, in the case $2\leq\gamma\leq3$ we have the global existence of the solutions to \eqref{eq:KGPKGP} from Proposition \ref{propal:solutglobalKGP} and thus we get \ref{item:point3mainth}.

\section{Proof of Theorem \ref{unTheorem:TH2mainth}}
\label{section:relatEul}
We give here the proof of the second Theorem, \ref{unTheorem:TH2mainth}, that states that the limit fluid system is in fact the true relativistic barotropic Euler system, given in \eqref{eq:relatEulintro} and the next section, up to a change of unknowns. We also make a comment on the non-relativistic limit of both the relativistic barotropic Euler system and the REP system \eqref{eq:EulerbaseREP}.  
\subsection{Equivalence of the REP system and the relativistic barotropic Euler}
\label{subsection:relatEullink}
We consider the usual relativistic Euler equation 
\begin{equation}
    \begin{cases}
    \label{eq:relatEul}
        \textbf{u}^\alpha\textbf{D}_\alpha \textbf{u}^\beta+(g^{\alpha\beta}+\frac{\textbf{u}^\alpha \textbf{u}^\beta}{c^2})\frac{\textbf{D}_\alpha p}{\mu+\frac{p}{c^2}}=0,\\
        \textbf{u}^\alpha\textbf{D}_\alpha\mu+\textbf{D}_\alpha \textbf{u}^\alpha(\mu+\frac{p}{c^2})=0,\\
        p=p(\mu),\\
        \textbf{u}^\alpha \textbf{u}_\alpha=-c^2,
    \end{cases}
\end{equation}
for $g$ the metric (here the Minkowski metric $g=Diag(-c^2,1,1,1)$), $\mu$ the energy density, $p$ the pressure, $\textbf{u}$ the four-velocity and $c$ the speed of light. The constant $c$ is not set up to $1$ because we are interested in the behaviour of this system for $c$ near $+\infty$. We refer to \cite{Gourgoulhon_2006}, \cite{Andersson_2021} and \cite{zbMATH05382453} for a general description of the relativistic fluid and to \cite{zbMATH07783487} and \cite{disconzi2023recent} for recent developments in the mathematical analysis.\\
The main difference between the usual relativistic Euler \eqref{eq:relatEul} and REP \eqref{eq:EulerbaseREP} is the role of the pressure and the potential. In \eqref{eq:relatEul}, we see that the gradient of the pressure (the density of force) acts only perpendicularly to the flow. Indeed the tensor 
\begin{equation}
 \label{eq:projector}
    \Pi_{\alpha\beta}=g_{\alpha\beta}+\frac{\textbf{u}_\alpha \textbf{u}_\beta}{c^2}
\end{equation}
is the projection on the orthogonal to $\textbf{u}$. The effect of the spacetime gradient of the potential $V'$ in \eqref{eq:EulerbaseREP} cannot be identified with the effect of a standard relativistic pressure. In comparison, in the non-relativistic case the pressure can clearly be identified with a potential. Moreover, the total energy density $\textbf{U}^0\rho$ is conserved in \eqref{eq:EulerbaseREP} whereas the energy $\textbf{u}^0\mu$ of \eqref{eq:relatEul} is not, and the normalization of $\textbf{U}$ and $\textbf{u}$ are different. Nonetheless, the REP system \eqref{eq:EulerbaseREP} and the usual relativistic Euler \eqref{eq:relatEul} are equivalent if we look at the right quantities.\\
\begin{propal}
 \label{propal:equivrelatEul}
    Let $(\textbf{U},\rho)$ be a solution to \eqref{eq:EulerbaseREP}, then the quantities 
    \begin{align}
    \label{eq:scaling1relatEuler}
        &\mu=\rho+\frac{(\gamma+1)V(\rho)}{c^2}, &\textbf{u}=\textbf{U}\Gamma,
    \end{align}
    with the scaling factor $\Gamma=\frac{1}{\sqrt{1+\frac{2V'}{c^2}}}$ and the pressure $p=(\gamma-1)V(\rho)$,  
    are solutions to \eqref{eq:relatEul}. We also have 
      \begin{align}
      \label{eq:scaling2relatEuler}
        &\textbf{J}=\textbf{U}\rho=\textbf{u}(\mu+\frac{p}{c^2})\Gamma, &\rho=(\mu+\frac{p}{c^2})\Gamma^2.
    \end{align}
\end{propal}
\begin{proof}
    First, we have 
    \begin{align*}
         &\textbf{u}^\alpha\textbf{D}_\alpha \textbf{u}^\beta=\Gamma^2 \textbf{U}^\alpha\textbf{D}_\alpha \textbf{U}^\beta+\Gamma \textbf{U}^\alpha \textbf{U}^\beta \textbf{D}_\alpha\Gamma=\Gamma^2(-\textbf{D}^\beta V'-\Gamma^2 \textbf{U}^\alpha \textbf{U}^\beta \frac{\textbf{D}_\alpha V'}{c^2})\\
         &=-(g^{\alpha\beta}+\frac{\textbf{u}^\alpha \textbf{u}^\beta}{c^2})\frac{\textbf{D}_\alpha V'}{1+\frac{2V'}{c^2}}=-(g^{\alpha\beta}+\frac{\textbf{u}^\alpha \textbf{u}^\beta}{c^2})\frac{(\gamma-1)V'\textbf{D}_\alpha\rho}{\rho+\frac{2\gamma V}{c^2}}\\
         &=-(g^{\alpha\beta}+\frac{\textbf{u}^\alpha \textbf{u}^\beta}{c^2})\frac{\textbf{D}_\alpha p}{\mu+\frac{p}{c^2}},\\
    \end{align*}
    and then 
        \begin{align*}
         &\textbf{u}^\alpha\textbf{D}_\alpha \mu=\textbf{u}^\alpha\textbf{D}_\alpha \rho +\textbf{u}^\alpha\textbf{D}_\alpha \rho (\gamma+1)\frac{V'}{c^2}=\frac{1}{\Gamma^2}\textbf{u}^\alpha\textbf{D}_\alpha \rho +\textbf{u}^\alpha\textbf{D}_\alpha \rho (\gamma-1)\frac{V'}{c^2}\\
         &=\frac{1}{\Gamma^2}(\textbf{u}^\alpha\textbf{D}_\alpha \rho +\textbf{u}^\alpha\textbf{D}_\alpha \rho\frac{V''}{c^2}\rho\Gamma^2)=\frac{1}{\Gamma^2}(-\textbf{D}_\alpha \textbf{U}^\alpha\Gamma\rho -\textbf{U}^\alpha\textbf{D}_\alpha(\Gamma)\rho)\\
         &=(1+\frac{2V'}{c^2})(-\textbf{D}_\alpha \textbf{u}^\alpha\rho)=-\textbf{D}_\alpha \textbf{u}^\alpha(\mu+\frac{p}{c^2}).\\
    \end{align*}
\end{proof}
\begin{cor}
\label{cor:theoremrelatiEul}
    The Theorem \ref{unTheorem:TH2mainth} is a direct corollary of Theorem \ref{unTheorem:TH1mainth} and Proposition \ref{propal:equivrelatEul}.
\end{cor}
\subsection{Non-relativistic limit}
\label{subsection:NRlimitEul}
We want to look at the non-relativistic limit at least formally. Let $(\textbf{u}_c,\mu_c)$ be the solution to \eqref{eq:relatEul}, traditionally, we write\footnote{We note $u^0_c$ the time component of $\textbf{u}_c$ and $u_c$ its space components.} $(u^0_c,u_c)$ as $(\theta_c,\theta_c v_c)$ where $\theta_c=\frac{1}{\sqrt{1-\frac{v_c^2}{c^2 }}}$ is the Lorentz factor and $v_c$ the non-relativistic velocity. Formally, as $c\to \infty$, we are expecting 
\begin{align}
\label{eq:NRlimitEul}
    &\lim_{c\to \infty} (u_c^0,u_c,\mu_c,p_c)=(u_\infty^0,u_\infty,\mu_\infty,p_\infty)=(1,v_\infty,\mu_\infty,p_\infty),
\end{align}
where $(v_\infty,\mu_\infty)$ is solution to the non-relativistic Euler equations and where $p_\infty$ is the corresponding pressure (see for example \cite{zbMATH07119817,zbMATH07531778,MAI2022336,zbMATH05077591,zbMATH00840326,zbMATH01309659,zbMATH07114642} for a rigorous proof of this convergence in dimension 1).\\ Moreover, with $(\textbf{u}_c,\mu_c)$ given by \ref{propal:equivrelatEul}, we clearly see that 
\begin{align*}
    &p_c=(\gamma-1)V(\rho_c), &\Gamma_c=1+O(c^{-2}),\\
    &\textbf{u}_c=\textbf{U}_c+O(c^{-2}), &\mu_c=\rho_c+O(c^{-2}),
\end{align*}
where we also index the solution to \eqref{eq:EulerbaseREP} with a $c$ to show the dependence. Thus, we have heuristically that 
\begin{align*}
    &\lim_{c\to \infty} (U_c^0,U_c,\rho_c,(\gamma-1)V(\rho_c))=(1,v_\infty,\mu_\infty,p_\infty),
\end{align*}
which is consistent with the non-relativistic semi-classical limit of the massive nonlinear Klein-Gordon toward the classical compressible Euler given in \cite{zbMATH06101438} and presented in the diagram of section \ref{subsection:comments}. We have 
\begin{align*}
    &(J^\varepsilon_c)^0\xrightarrow{s.c.}U^0_c\rho_c=u^0_c(\mu_c+\frac{p_c}{c^2})\Gamma_c\xrightarrow{n.r.}\mu_\infty,\\
      &J^\varepsilon_c\xrightarrow{s.c.}U_c\rho_c=u_c(\mu_c+\frac{p_c}{c^2})\Gamma_c\xrightarrow{n.r.}v_\infty\mu_\infty,\\
        &|\Phi_c^\varepsilon|^2\xrightarrow{s.c.}\rho_c=(\mu_c+\frac{p_c}{c^2})\Gamma_c^2\xrightarrow{n.r.}\mu_\infty,
\end{align*}
where the semi-classical limit is rigorous and the non-relativistic limit is heuristic. 
\appendix
\section{Local well-posed for the REP system}
\label{section:appendix2}
We give the proof of Proposition \ref{propal:wellposedEulerEP} which states the local well-posedness of the REP system \eqref{eq:EulerbaseREP}, that we reintroduce here 
\begin{equation}
\label{eq:EulerbaseREPap}
\begin{cases}
\textbf{U}^\alpha\textbf{D}_\alpha \textbf{U}_\beta+\textbf{D}_\beta V'(\rho)=0,\\
\textbf{U}^\alpha\textbf{D}_\alpha \rho+\textbf{D}_\alpha \textbf{U}^\alpha\rho=0,\\
\textbf{U}^\alpha \textbf{U}_\alpha+2V'(\rho)=-1.
\end{cases}
\end{equation}
\begin{proof}[Proof of Proposition \ref{propal:wellposedEulerEP}]
We observe that the system \eqref{eq:EulerbaseREPap} is overdetermined in this form because of the normalization condition. We forget for now about this normalization, we use it in a later argument. 
  The goal is to put the system under a quasilinear symmetrized hyperbolic form and apply the classical well posedness result using energy methods. We can observe that the system is almost symmetric, but the natural symmetrizer $Diag(\rho,\rho,\rho,\rho,V'')$ is not uniformly coercive. It is a classical problem in fluid dynamics, it corresponds to the presence of vacuum. As done in \cite{zbMATH05243173} and originally in \cite{zbMATH04039498}, we chose to write the system \eqref{eq:EulerbaseREPap} in terms of $\textbf{U}$ and $f=\sqrt{\rho}^{\gamma-1}$, thus, we need to assume the regularity of $\sqrt{\rho}$. The well-prepared initial data have this feature, see Definition \ref{defi:wellprepREP}. Written with respect to $(\textbf{U},f=\sqrt{\rho}^{\gamma-1})$, the system \eqref{eq:EulerbaseREPap} (without the normalization) is equivalent to 
  \begin{equation}
  \begin{cases}
  \label{eq:goodEuler2EP}
       \textbf{U}^\alpha\textbf{D}_\alpha \textbf{U}_\beta+2f\textbf{D}_\beta f=0,\\
        \frac{4}{\gamma-1}\textbf{U}^\alpha\textbf{D}_\alpha f+2\textbf{D}_\alpha \textbf{U}^\alpha f=0,\\
  \end{cases}
  \end{equation}
  and with the argument of Remark \ref{rem:rhoremarkREP}, we see that $\rho$ and $f$ stay positive, so the relation between $f$ and $\rho$ is invertible. The system \eqref{eq:goodEuler2EP} is symmetric, we have 
    \begin{equation}
    \label{eq:goodEuler3EP}
       A^\alpha\textbf{D}_\alpha\begin{pmatrix}
    \textbf{U} \\
    f
\end{pmatrix}=0,
  \end{equation}
  where the $A^\alpha$ are symmetric matrices. Nonetheless, $A^0$ does not fit the coercivity criteria ($A^0(X,X)>c|X|^2$ for some $c>0$) to apply the energy method (in fact, $A^0$ is not positive definite). Instead of only looking at the transport equations \eqref{eq:goodEuler3EP} and then use the fact that the normalization is transported by them (from Proposition \ref{propal:propagwellprepREP}), we use actively the normalization 
  \begin{align}
  \label{eq:normalU0}
U^0=\textbf{U}^0=\sqrt{\textbf{U}^i\textbf{U}_i+1+2f^2}=\sqrt{U^iU_i+1+2f^2}.
  \end{align}
We search for solutions to the reduced system 
  \begin{equation}
  \label{eq:goodEuler4EP}
  \begin{cases}
       B^\alpha\textbf{D}_\alpha\begin{pmatrix}
    U \\
    f
\end{pmatrix}=0,\\
 U^0=\textbf{U}^0(U,f),
  \end{cases}
  \end{equation}
  with 
\begin{align*}
    &B^0=
\begin{pmatrix} 
	U^0 & 0 & 0 & 0\\
	0 & U^0  & 0 & 0 \\
	0 & 0 & U^0  & 0 \\
	0 & 0 & 0 & 4U^0(\frac{1}{\gamma-1}+\frac{f^2}{(U^0)^2})  \\
	\end{pmatrix}, &B^i=
\begin{pmatrix} 
	U^i & 0 & 0 & 2f\delta^i_1\\
	0 & U^i  & 0 & 2f\delta^i_2 \\
	0 & 0 & U^i  & 2f\delta^i_3 \\
	2fM^i_{~1} & 2fM^i_{~2} & 2fM^i_{~3} & 4U^i(\frac{1}{\gamma-1}-\frac{f^2}{(U^0)^2})  \\
	\end{pmatrix},
\end{align*}
where $M_{ij}=\delta_{ij}-\frac{U_iU_j}{(U^0)^2}$.
  One can check that the solutions to the previous system \eqref{eq:goodEuler4EP} are equivalent to the solutions to \eqref{eq:goodEuler3EP}. The matrix $M$ is positive definite. In particular, its eigenvalues are $1,1$ and $1-\frac{U_iU^i}{(U^0)^2}=\frac{1+2f^2}{1+U^iU_i+2f^2}>c$ for some $c>0$. Indeed, $U^i\leq ||U_i||_{H^2}\leq C_0$ and so $(U^iU_i)$ stay finite. We define the $4\times4$ symmetrizer 
  \[S=
\begin{pmatrix} 
	M & 0 \\
    0 & 1
	\end{pmatrix}.
	\quad
	\]
 We see that the $SB^\alpha$ are symmetric matrices with $SB^0$ uniformly coercive ($U^0\geq 1$ and $4(\frac{1}{\gamma-1}+\frac{f^2}{(U^0)^2})\geq \frac{4}{\gamma-1}$). We deduce that the system \eqref{eq:goodEuler4EP} is well posed for initial data $(U,f)$ in $H^n$ for $n\geq 4$, this is the case for well-prepared initial data. Then, we calculate $U^0$ with the equation \eqref{eq:normalU0} and we get all the components of $\textbf{U}$.
\end{proof}
\printbibliography

@Article{zbMATH06101438,
 Author = {Lin, Chi-Kun and Wu, Kung-Chien},
 Title = {Hydrodynamic limits of the nonlinear {Klein}-{Gordon} equation},
 FJournal = {Journal de Math{\'e}matiques Pures et Appliqu{\'e}es. Neuvi{\`e}me S{\'e}rie},
 Journal = {J. Math. Pures Appl. (9)},
 ISSN = {0021-7824},
 Volume = {98},
 Number = {3},
 Pages = {328--345},
 Year = {2012},
 Language = {English},
 DOI = {10.1016/j.matpur.2012.02.002},
 zbMATH = {6101438},
 Zbl = {1253.35110}
}

@Article{zbMATH06039446,
 Author = {Lin, Chi-Kun and Wu, Kung-Chien},
 Title = {On the fluid dynamical approximation to the nonlinear {Klein}-{Gordon} equation},
 FJournal = {Discrete and Continuous Dynamical Systems},
 Journal = {Discrete Contin. Dyn. Syst.},
 ISSN = {1078-0947},
 Volume = {32},
 Number = {6},
 Pages = {2233--2251},
 Year = {2012},
 Language = {English},
 DOI = {10.3934/dcds.2012.32.2233},
 zbMATH = {6039446},
 Zbl = {1251.35041}
}

@Article{zbMATH05782734,
 Author = {Lin, Chi-Kun and Wu, Kung-Chien},
 Title = {Singular limits of the {Klein}-{Gordon} equation},
 FJournal = {Archive for Rational Mechanics and Analysis},
 Journal = {Arch. Ration. Mech. Anal.},
 ISSN = {0003-9527},
 Volume = {197},
 Number = {2},
 Pages = {689--711},
 Year = {2010},
 Language = {English},
 DOI = {10.1007/s00205-010-0324-8},
 zbMATH = {5782734},
 Zbl = {1198.35259}
}

@Article{zbMATH01442880,
 Author = {Brenier, Y.},
 Title = {Convergence of the {Vlasov}-{Poisson} system to the incompressible {Euler} equations},
 FJournal = {Communications in Partial Differential Equations},
 Journal = {Commun. Partial Differ. Equations},
 ISSN = {0360-5302},
 Volume = {25},
 Number = {3-4},
 Pages = {737--754},
 Year = {2000},
 Language = {English},
 DOI = {10.1080/03605300008821529},
 zbMATH = {1442880},
 Zbl = {0970.35110}
}

@Article{zbMATH01837395,
 Author = {Masmoudi, Nader and Nakanishi, Kenji},
 Title = {From nonlinear {Klein}-{Gordon} equation to a system of coupled nonlinear {Schr{\"o}dinger} equations},
 FJournal = {Mathematische Annalen},
 Journal = {Math. Ann.},
 ISSN = {0025-5831},
 Volume = {324},
 Number = {2},
 Pages = {359--389},
 Year = {2002},
 Language = {English},
 DOI = {10.1007/s00208-002-0342-4},
 zbMATH = {1837395},
 Zbl = {1008.35071}
}

@Article{zbMATH01747314,
 Author = {Machihara, Shuji and Nakanishi, Kenji and Ozawa, Tohru},
 Title = {Nonrelativistic limit in the energy space for nonlinear {Klein}-{Gordon} equations},
 FJournal = {Mathematische Annalen},
 Journal = {Math. Ann.},
 ISSN = {0025-5831},
 Volume = {322},
 Number = {3},
 Pages = {603--621},
 Year = {2002},
 Language = {English},
 DOI = {10.1007/s002080200008},
 zbMATH = {1747314},
 Zbl = {0991.35080}
}

@misc{lei2023nonrelativistic,
      title={Non-relativistic limit for the cubic nonlinear Klein-Gordon equations}, 
      author={Zhen Lei and Yifei Wu},
      year={2023},
      eprint={2309.10235},
      archivePrefix={arXiv},
      primaryClass={math.AP}
}

@Book{zbMATH05243173,
 Author = {Carles, R{\'e}mi},
 Title = {Semi-classical analysis for nonlinear {Schr{\"o}dinger} equations},
 ISBN = {978-981-279-312-6},
 Year = {2008},
 Publisher = {Hackensack, NJ: World Scientific},
 Language = {English},
 zbMATH = {5243173},
 Zbl = {1153.35070}
}

@Article{zbMATH01132286,
 Author = {Grenier, E.},
 Title = {Semiclassical limit of the nonlinear {Schr{\"o}dinger} equation in small time},
 FJournal = {Proceedings of the American Mathematical Society},
 Journal = {Proc. Am. Math. Soc.},
 ISSN = {0002-9939},
 Volume = {126},
 Number = {2},
 Pages = {523--530},
 Year = {1998},
 Language = {English},
 DOI = {10.1090/S0002-9939-98-04164-1},
 zbMATH = {1132286},
 Zbl = {0910.35115}
}

@article{Gourgoulhon_2006,
   title={An introduction to relativistic hydrodynamics},
   volume={21},
   ISSN={1638-1963},
   url={http://dx.doi.org/10.1051/eas:2006106},
   DOI={10.1051/eas:2006106},
   journal={EAS Publications Series},
   publisher={EDP Sciences},
   author={Gourgoulhon, E.},
   year={2006},
   pages={43–79} }

@article{Andersson_2021,
   title={Relativistic fluid dynamics: physics for many different scales},
   volume={24},
   ISSN={1433-8351},
   url={http://dx.doi.org/10.1007/s41114-021-00031-6},
   DOI={10.1007/s41114-021-00031-6},
   number={1},
   journal={Living Reviews in Relativity},
   publisher={Springer Science and Business Media LLC},
   author={Andersson, Nils and Comer, Gregory L.},
   year={2021},
   month=jun }

@Book{zbMATH05382453,
 Author = {Choquet-Bruhat, Yvonne},
 Title = {General relativity and the {Einstein} equations.},
 FSeries = {Oxford Mathematical Monographs},
 Series = {Oxford Math. Monogr.},
 ISBN = {978-0-19-923072-3},
 Year = {2009},
 Publisher = {Oxford: Oxford University Press},
 Language = {English},
 zbMATH = {5382453},
 Zbl = {1157.83002}
}

@misc{disconzi2023recent,
      title={Recent developments in mathematical aspects of relativistic fluids}, 
      author={Marcelo M. Disconzi},
      year={2023},
      eprint={2308.09844},
      archivePrefix={arXiv},
      primaryClass={math.AP}
}

@Article{zbMATH07783487,
 Author = {Abbrescia, Leonardo and Speck, Jared},
 Title = {The relativistic {Euler} equations: {ESI} notes on their geo-analytic structures and implications for shocks in {{\(1D\)}} and multi-dimensions},
 FJournal = {Classical and Quantum Gravity},
 Journal = {Classical Quantum Gravity},
 ISSN = {0264-9381},
 Volume = {40},
 Number = {24},
 Pages = {79},
 Note = {Id/No 243001},
 Year = {2023},
 Language = {English},
 DOI = {10.1088/1361-6382/ad059a},
 zbMATH = {7783487}
}

@Article{zbMATH07119817,
 Author = {Mai, La-Su and Li, Hai-Liang and Marcati, Pierangelo},
 Title = {Non-relativistic limit analysis of the {Chandrasekhar}-{Thorne} relativistic {Euler} equations with physical vacuum},
 FJournal = {M\(^3\)AS. Mathematical Models \& Methods in Applied Sciences},
 Journal = {Math. Models Methods Appl. Sci.},
 ISSN = {0218-2025},
 Volume = {29},
 Number = {3},
 Pages = {531--579},
 Year = {2019},
 Language = {English},
 DOI = {10.1142/S0218202519500155},
 zbMATH = {7119817},
 Zbl = {1428.35331}
}

@Article{zbMATH07531778,
 Author = {Chen, Gui-Qiang G. and Schrecker, Matthew R. I.},
 Title = {Global entropy solutions and {Newtonian} limit for the relativistic {Euler} equations},
 FJournal = {Annals of PDE},
 Journal = {Ann. PDE},
 ISSN = {2524-5317},
 Volume = {8},
 Number = {1},
 Pages = {53},
 Note = {Id/No 10},
 Year = {2022},
 Language = {English},
 DOI = {10.1007/s40818-022-00123-8},
 zbMATH = {7531778},
 Zbl = {1490.35260}
}

@article{MAI2022336,
title = {Newtonian limit for the relativistic Euler-Poisson equations with vacuum},
journal = {Journal of Differential Equations},
volume = {313},
pages = {336-381},
year = {2022},
issn = {0022-0396},
doi = {https://doi.org/10.1016/j.jde.2022.01.003},
url = {https://www.sciencedirect.com/science/article/pii/S0022039622000031},
author = {La-Su Mai and Ming Mei}
}

@incollection{METIVIER2009169,
title = {Chapter 3 - The Mathematics of Nonlinear Optics},
editor = {C.M. Dafermos and Milan Pokorný},
series = {Handbook of Differential Equations: Evolutionary Equations},
publisher = {North-Holland},
volume = {5},
pages = {169-313},
year = {2009},
booktitle = {Handbook of Differential Equations},
issn = {1874-5717},
doi = {https://doi.org/10.1016/S1874-5717(08)00210-7},
url = {https://www.sciencedirect.com/science/article/pii/S1874571708002107},
author = {Guy Métivier}
}

@inproceedings{Rauch1999LecturesOG,
  title={Lectures on Geometric Optics},
  author={Jeffrey Rauch and Markus Keel},
  year={1999},
  url={https://api.semanticscholar.org/CorpusID:118666352}
}

@article{Han_Kwan_2016,
   title={Ill-Posedness of the Hydrostatic Euler and Singular Vlasov Equations},
   volume={221},
   ISSN={1432-0673},
   url={http://dx.doi.org/10.1007/s00205-016-0985-z},
   DOI={10.1007/s00205-016-0985-z},
   number={3},
   journal={Archive for Rational Mechanics and Analysis},
   publisher={Springer Science and Business Media LLC},
   author={Han-Kwan, Daniel and Nguyen, Toan T.},
   year={2016},
   month=feb, pages={1317–1344} }

@Article{zbMATH00203353,
 Author = {Markowich, Peter A. and Mauser, Norbert J.},
 Title = {The classical limit of a self-consistent quantum-{Vlasov} equation in 3D},
 FJournal = {M\(^3\)AS. Mathematical Models \& Methods in Applied Sciences},
 Journal = {Math. Models Methods Appl. Sci.},
 ISSN = {0218-2025},
 Volume = {3},
 Number = {1},
 Pages = {109--124},
 Year = {1993},
 Language = {English},
 DOI = {10.1142/S0218202593000072},
 zbMATH = {203353},
 Zbl = {0772.35061}
}

@inproceedings{Mauser2002SemiclassicalLO,
  title={Semi-classical limits of Schr{\"o}dinger-Poisson systems via Wigner transforms},
  author={Norbert J. Mauser},
  year={2002},
  url={https://api.semanticscholar.org/CorpusID:15405057}
}

@misc{carles2023nonlineareffectsmultiphasewkb,
      title={On nonlinear effects in multiphase WKB analysis for the nonlinear Schrodinger equation}, 
      author={Remi Carles},
      year={2023},
      eprint={2309.03597},
      archivePrefix={arXiv},
      primaryClass={math.AP},
      url={https://arxiv.org/abs/2309.03597}, 
}

@article{Carles_2009,
   title={On the time evolution of Wigner measures for Schrödinger equations},
   volume={8},
   ISSN={1553-5258},
   url={http://dx.doi.org/10.3934/cpaa.2009.8.559},
   DOI={10.3934/cpaa.2009.8.559},
   number={2},
   journal={Communications on Pure \&amp; Applied Analysis},
   publisher={American Institute of Mathematical Sciences (AIMS)},
   author={Carles, Rémi and Fermanian-Kammerer, Clotilde and J. Mauser, Norbert and Peter Stimming, Hans},
   year={2009},
   pages={559–585} }

@Article{zbMATH00482230,
 Author = {Lions, Pierre Louis and Paul, Thierry},
 Title = {On {Wigner} measures},
 FJournal = {Revista Matem{\'a}tica Iberoamericana},
 Journal = {Rev. Mat. Iberoam.},
 ISSN = {0213-2230},
 Volume = {9},
 Number = {3},
 Pages = {553--618},
 Year = {1993},
 Language = {French},
 DOI = {10.4171/RMI/143},
 zbMATH = {482230},
 Zbl = {0801.35117}
}

@Article{zbMATH05000052,
 Author = {Lin, Fanghua and Zhang, Ping},
 Title = {Semiclassical limit of the {Gross}-{Pitaevskii} equation in an exterior domain},
 FJournal = {Archive for Rational Mechanics and Analysis},
 Journal = {Arch. Ration. Mech. Anal.},
 ISSN = {0003-9527},
 Volume = {179},
 Number = {1},
 Pages = {79--107},
 Year = {2006},
 Language = {English},
 DOI = {10.1007/s00205-005-0383-4},
 zbMATH = {5000052},
 Zbl = {1079.76016}
}

@Article{zbMATH01987564,
 Author = {Zhang, Ping},
 Title = {Wigner measure and the semiclassical limit of {Schr{\"o}dinger}-{Poisson} equations},
 FJournal = {SIAM Journal on Mathematical Analysis},
 Journal = {SIAM J. Math. Anal.},
 ISSN = {0036-1410},
 Volume = {34},
 Number = {3},
 Pages = {700--718},
 Year = {2003},
 Language = {English},
 DOI = {10.1137/S0036141001393407},
 zbMATH = {1987564},
 Zbl = {1032.35132}
}

@misc{luk2014introduction,
  title={Introduction to nonlinear wave equations},
  author={Luk, Jonathan},
  year={2014}
}

@article{wang2015lectures,
  title={Lectures on nonlinear wave equations},
  author={Wang, Qian},
  journal={lecture notes, University of Oxford},
  year={2015}
}

@Book{zbMATH00928933,
 Author = {Lions, Pierre-Louis},
 Title = {Mathematical topics in fluid mechanics. {Vol}. 1: {Incompressible} models},
 FSeries = {Oxford Lecture Series in Mathematics and its Applications},
 Series = {Oxf. Lect. Ser. Math. Appl.},
 Volume = {3},
 ISBN = {0-19-851487-5},
 Year = {1996},
 Publisher = {Oxford: Clarendon Press},
 Language = {English},
 zbMATH = {928933},
 Zbl = {0866.76002}
}

@Book{zbMATH06039454,
 Author = {Rauch, Jeffrey},
 Title = {Hyperbolic partial differential equations and geometric optics},
 FSeries = {Graduate Studies in Mathematics},
 Series = {Grad. Stud. Math.},
 ISSN = {1065-7338},
 Volume = {133},
 ISBN = {978-0-8218-7291-8},
 Year = {2012},
 Publisher = {Providence, RI: American Mathematical Society (AMS)},
 Language = {English},
 zbMATH = {6039454},
 Zbl = {1252.35004}
}

@Article{zbMATH04039498,
 Author = {Makino, Tetu and Ukai, Seiji and Kawashima, Shuichi},
 Title = {Sur la solution {\`a} support compact de l'equation d'{Euler} compressible. ({Solutions} with compact support of the compressible {Euler} equation)},
 FJournal = {Japan Journal of Applied Mathematics},
 Journal = {Japan J. Appl. Math.},
 ISSN = {0910-2043},
 Volume = {3},
 Pages = {249--257},
 Year = {1986},
 Language = {French},
 DOI = {10.1007/BF03167100},
 zbMATH = {4039498},
 Zbl = {0637.76065}
}

@Article{zbMATH01877186,
 Author = {Puel, Marjolaine},
 Title = {Convergence of the {Schr{\"o}dinger}-{Poisson} system to the incompressible {Euler} equations.},
 FJournal = {Communications in Partial Differential Equations},
 Journal = {Commun. Partial Differ. Equations},
 ISSN = {0360-5302},
 Volume = {27},
 Number = {11-12},
 Pages = {2311--2331},
 Year = {2002},
 Language = {English},
 DOI = {10.1081/PDE-120016159},
 zbMATH = {1877186},
 Zbl = {1040.35076}
}

@Article{zbMATH01970480,
 Author = {Puel, Marjolaine},
 Title = {Convergence of the {Schr{\"o}dinger}-{Poisson} system to the {Euler} equations under the influence of a large magnetic field.},
 FJournal = {M2AN. Mathematical Modelling and Numerical Analysis. ESAIM, European Series in Applied and Industrial Mathematics},
 Journal = {M2AN, Math. Model. Numer. Anal.},
 ISSN = {0764-583X},
 Volume = {36},
 Number = {6},
 Pages = {1071--1090},
 Year = {2002},
 Language = {English},
 DOI = {10.1051/m2an:2003006},
 URL = {https://eudml.org/doc/194140},
 zbMATH = {1970480},
 Zbl = {1137.76836}
}

@Article{zbMATH07570763,
 Author = {Golse, Fran{\c{c}}ois and Paul, Thierry},
 Title = {Mean-field and classical limit for the {{\(N\)}}-body quantum dynamics with {Coulomb} interaction},
 FJournal = {Communications on Pure and Applied Mathematics},
 Journal = {Commun. Pure Appl. Math.},
 ISSN = {0010-3640},
 Volume = {75},
 Number = {6},
 Pages = {1332--1376},
 Year = {2022},
 Language = {English},
 DOI = {10.1002/cpa.21986},
 zbMATH = {7570763},
 Zbl = {1502.81072}
}

@misc{yang2024semiclassicallimitpaulipoisswelleulerpoisswell,
      title={The semiclassical limit from the Pauli-Poisswell to the Euler-Poisswell system by WKB methods}, 
      author={Changhe Yang and Norbert J. Mauser and Jakob Möller},
      year={2024},
      eprint={2304.06660},
      archivePrefix={arXiv},
      primaryClass={math.AP},
      url={https://arxiv.org/abs/2304.06660}, 
}

@article{zbMATH03871951,
 author = {Tsutsumi, Masayoshi},
 title = {Nonrelativistic approximation of nonlinear {Klein}-{Gordon} equations in two space dimensions},
 fjournal = {Nonlinear Analysis. Theory, Methods \& Applications},
 journal = {Nonlinear Anal., Theory Methods Appl.},
 issn = {0362-546X},
 volume = {8},
 pages = {637--643},
 year = {1984},
 language = {English},
 doi = {10.1016/0362-546X(84)90008-7},
 zbMATH = {3871951},
 Zbl = {0547.35111}
}

@article{NAJMAN1990217,
title = {The nonrelativistic limit of the nonlinear Klein-Gordon equation},
journal = {Nonlinear Analysis: Theory, Methods \& Applications},
volume = {15},
number = {3},
pages = {217-228},
year = {1990},
issn = {0362-546X},
doi = {https://doi.org/10.1016/0362-546X(90)90158-D},
url = {https://www.sciencedirect.com/science/article/pii/0362546X9090158D},
author = {Branko Najman}
}

@article{zbMATH01650142,
 author = {Machihara, Sh{\=u}ji},
 title = {The nonrelativistic limit of the nonlinear {Klein}-{Gordon} equation},
 fjournal = {RIMS Kokyuroku},
 journal = {RIMS Kokyuroku},
 issn = {1880-2818},
 volume = {1162},
 pages = {80--90},
 year = {2000},
 language = {English},
 zbMATH = {1650142},
 Zbl = {0969.35554}
}

@article{zbMATH06303453,
 author = {Faou, Erwan and Schratz, Katharina},
 title = {Asymptotic preserving schemes for the {Klein}-{Gordon} equation in the non-relativistic limit regime},
 fjournal = {Numerische Mathematik},
 journal = {Numer. Math.},
 issn = {0029-599X},
 volume = {126},
 number = {3},
 pages = {441--469},
 year = {2014},
 language = {English},
 doi = {10.1007/s00211-013-0567-z},
 zbMATH = {6303453},
 Zbl = {1314.65131}
}

@misc{pasquali2018dynamicsnonlinearkleingordonequation,
      title={Dynamics of the nonlinear Klein-Gordon equation in the nonrelativistic limit, I}, 
      author={Stefano Pasquali},
      year={2018},
      eprint={1703.01609},
      archivePrefix={arXiv},
      primaryClass={math.AP},
      url={https://arxiv.org/abs/1703.01609}, 
}

@misc{bambusi2025nonrelativisticlimitnonlinear,
      title={Non relativistic limit of the nonlinear Klein-Gordon equation: Uniform in time convergence of KAM solutions}, 
      author={Dario Bambusi and Andrea Belloni and Filippo Giuliani},
      year={2025},
      eprint={2501.17691},
      archivePrefix={arXiv},
      primaryClass={math-ph},
      url={https://arxiv.org/abs/2501.17691}, 
}

@article{zbMATH07114642,
 author = {Mai, La-Su and Cao, Xiaoting},
 title = {Nonrelativistic limits for the 1D relativistic {Euler} equations with physical vacuum},
 fjournal = {ZAMP. Zeitschrift f{\"u}r angewandte Mathematik und Physik},
 journal = {Z. Angew. Math. Phys.},
 issn = {0044-2275},
 volume = {70},
 number = {5},
 pages = {19},
 note = {Id/No 145},
 year = {2019},
 language = {English},
 doi = {10.1007/s00033-019-1189-9},
 zbMATH = {7114642},
 Zbl = {1428.35330}
}

@article{zbMATH01309659,
 author = {Min, Lu and Ukai, Seiji},
 title = {Non-relativistic global limits of weak solutions of the relativistic {Euler} equation},
 fjournal = {Journal of Mathematics of Kyoto University},
 journal = {J. Math. Kyoto Univ.},
 issn = {0023-608X},
 volume = {38},
 number = {3},
 pages = {525--537},
 year = {1998},
 language = {English},
 doi = {10.1215/kjm/1250518065},
 zbMATH = {1309659},
 Zbl = {0939.35154}
}

@article{zbMATH00840326,
 author = {Makino, Tetu and Ukai, Seiji},
 title = {Local smooth solutions of the relativistic {Euler} equation. {II}},
 fjournal = {Kodai Mathematical Journal},
 journal = {Kodai Math. J.},
 issn = {0386-5991},
 volume = {18},
 number = {2},
 pages = {365--375},
 year = {1995},
 language = {English},
 doi = {10.2996/kmj/1138043432},
 zbMATH = {840326},
 Zbl = {0870.76103}
}

@article{zbMATH05077591,
 author = {Li, Yachun and Geng, Yongcai},
 title = {Non-relativistic global limits of entropy solutions to the isentropic relativistic {Euler} equations},
 fjournal = {ZAMP. Zeitschrift f{\"u}r angewandte Mathematik und Physik},
 journal = {Z. Angew. Math. Phys.},
 issn = {0044-2275},
 volume = {57},
 number = {6},
 pages = {960--983},
 year = {2006},
 language = {English},
 doi = {10.1007/s00033-006-0059-4},
 zbMATH = {5077591},
 Zbl = {1161.35321}
}

@article{zbMATH05503664,
	author = {Alazard, Thomas and Carles, R{\'e}mi},
	doi = {10.1007/s00208-008-0276-6},
	fjournal = {Mathematische Annalen},
	issn = {0025-5831},
	journal = {Math. Ann.},
	language = {English},
	number = {2},
	pages = {397--420},
	title = {Loss of regularity for supercritical nonlinear {Schr{\"o}dinger} equations},
	volume = {343},
	year = {2009},
	zbl = {1161.35047},
	zbmath = {5503664}}

@article{zbMATH03875887,
 author = {Ginibre, J. and Velo, G.},
 title = {The global {Cauchy} problem for the nonlinear {Klein}-{Gordon} equation},
 fjournal = {Mathematische Zeitschrift},
 journal = {Math. Z.},
 issn = {0025-5874},
 volume = {189},
 pages = {487--505},
 year = {1985},
 language = {English},
 doi = {10.1007/BF01168155},
 url = {https://eudml.org/doc/173596},
 zbMATH = {3875887},
 Zbl = {0549.35108}
}

@article{zbMATH04202876,
 author = {Struwe, Michael},
 title = {Globally regular solutions to the {\({{\mathcal U}}^ 5\)} {Klein}-{Gordon} equation},
 fjournal = {Annali della Scuola Normale Superiore di Pisa. Classe di Scienze. Serie IV},
 journal = {Ann. Sc. Norm. Super. Pisa, Cl. Sci., IV. Ser.},
 issn = {0391-173X},
 volume = {15},
 number = {3},
 pages = {495--513},
 year = {1988},
 language = {English},
 url = {https://eudml.org/doc/84039},
 zbMATH = {4202876},
 Zbl = {0728.35072}
}

@article{zbMATH00703984,
 author = {Shatah, Jalal and Struwe, Michael},
 title = {Well-posedness in the energy space for semilinear wave equations with critical growth},
 fjournal = {IMRN. International Mathematics Research Notices},
 journal = {Int. Math. Res. Not.},
 issn = {1073-7928},
 volume = {1994},
 number = {7},
 pages = {303--309},
 year = {1994},
 language = {English},
 doi = {10.1155/S1073792894000346},
 zbMATH = {703984},
 Zbl = {0830.35086}
}

@article{zbMATH00817689,
 author = {Kapitanski, Lev},
 title = {Global and unique weak solutions of nonlinear wave equations},
 fjournal = {Mathematical Research Letters},
 journal = {Math. Res. Lett.},
 issn = {1073-2780},
 volume = {1},
 number = {2},
 pages = {211--223},
 year = {1994},
 language = {English},
 doi = {10.4310/MRL.1994.v1.n2.a9},
 zbMATH = {817689},
 Zbl = {0841.35067}
}

@article{zbMATH00011219,
 author = {Grillakis, Manoussos G.},
 title = {Regularity and asymptotic behaviour of the wave equation with a critical nonlinearity},
 fjournal = {Annals of Mathematics. Second Series},
 journal = {Ann. Math. (2)},
 issn = {0003-486X},
 volume = {132},
 number = {3},
 pages = {485--509},
 year = {1990},
 language = {English},
 doi = {10.2307/1971427},
 zbMATH = {11219},
 Zbl = {0736.35067}
}

@article{zbMATH00168350,
 author = {Grillakis, Manoussos G.},
 title = {Regularity for the wave equation with a critical nonlinearity},
 fjournal = {Communications on Pure and Applied Mathematics},
 journal = {Commun. Pure Appl. Math.},
 issn = {0010-3640},
 volume = {45},
 number = {6},
 pages = {749--774},
 year = {1992},
 language = {English},
 doi = {10.1002/cpa.3160450604},
 zbMATH = {168350},
 Zbl = {0785.35065}
}

@book{zbMATH05035890,
 author = {Tao, Terence},
 title = {Nonlinear dispersive equations. {Local} and global analysis},
 fseries = {CBMS Regional Conference Series in Mathematics},
 series = {CBMS Reg. Conf. Ser. Math.},
 issn = {0160-7642},
 volume = {106},
 isbn = {0-8218-4143-2},
 year = {2006},
 publisher = {Providence, RI: American Mathematical Society (AMS)},
 language = {English},
 zbMATH = {5035890},
 Zbl = {1106.35001}
}

@article{Brenner1979,
author = {Brenner, Philip},
journal = {Mathematische Zeitschrift},
pages = {99-136},
title = {On the Existence of Global Smooth Solutions of Certain Semi-Linear Hyperbolic Equations.},
url = {http://eudml.org/doc/172841},
volume = {167},
year = {1979},
}

@article{Wahl1981,
author = {Wahl, Wolf von and Brenner, Philip},
journal = {Mathematische Zeitschrift},
pages = {87-122},
title = {Global Classical Solutions of Nonlinear Wave Equations.},
url = {http://eudml.org/doc/183633},
volume = {176},
year = {1981},
}
\end{document}